\documentclass[9pt]{amsart}

\usepackage{graphicx}

\def\R{{\mathbb R}}

\def\N{{\mathbb N}}

\def\Z{{\mathbb Z}}
\def\1{{1\!\!\!1}}
\def\Id{{\mathbb I}}

\def\E{{\mathbb E}}
\def\eps{{\epsilon}}
\def\P{{\mathbb P}}

\def\cal{\mathcal}

\def\supp{{\rm{supp}}}
\def\eps{\varepsilon}

\newcommand{\be}{\begin{equation}}
\newcommand{\ee}{\end{equation}}
\numberwithin{equation}{section}

\newtheorem{theorem}{Theorem}
\newtheorem{prop}{Proposition}[section]
\newtheorem{cor}{Corollary}[section]
\newtheorem{defi}{Definition}[section]
\newtheorem{lemma}{Lemma}[section]

\title{Harmonic functions of random walks in a semigroup via ladder heights.}
\author{Irina Ignatiouk-Robert}
\address{
{Universit\'e de Cergy-Pontoise,}
{D\'epartement de math\'ematiques,}
{2, Avenue Adolphe Chauvin,}
{95302 Cergy-Pontoise Cedex,}
{France}}
\date{\today}
\email{Irina.Ignatiouk@u-cergy.fr}

\date{Received: date / Accepted: date}

\begin{document}
\begin{abstract} We investigate  harmonic functions and the  convergence of the sequence of ratios $(\P_x(\tau_\vartheta {>} n)/\P_e(\tau_\vartheta {>} n))$  for a random walk on a countable group killed up on the time $\tau_\vartheta$ of the first exit from some semi-group with an identity element $e$. Several results  of  classical renewal theory  for one dimensional random walk killed  at the first exit from the positive half-line are extended to a multi-dimensional setting. For this purpose, an analogue of the ladder height process and the corresponding renewal function $V$ are  introduced.  The results are applied to multidimensional  random walks $(X(t))$ killed upon the times of first exit from a convex cone. Our approach combines large deviation estimates  and an extension of Choquet-Deny theory.
\end{abstract}
\maketitle

\section{Introduction}

An explicit description of all harmonic functions for a transient countable Markov chain is in general a non trivial problem.  Markov chains associated to  homogeneous random walks are one of the few examples when  such a complete description   can be obtained. This is done via the classical method of  Choquet-Deny theory,  see for example Sawyer~\cite{Sawyer}. 

For a non homogeneous transient Markov chain with a transition kernel $P$, the method of  Choquet-Deny theory does not work. To find all harmonic functions one has  either to solve the equation $Ph{=}h$ in a straightforward way via analytical methods, or to investigate the Martin boundary of the process, by identifying all possible limits of the Martin kernel. The classical references for the second approach are  Doob~\cite{Doob}, Hunt~\cite{Hunt} and Dynkin~\cite{Dynkin:01}. Both approaches are non-trivial and there are also few examples when all harmonic functions can be found. 

A simple but important example of a non-homogeneous Markov chain is a random walk on some group $({\cal E},\star)$ killed up on the time $\tau_\vartheta$ of the first exit from some semi-group $E{\subset}{\cal E}$. The harmonic functions of such a random walk are related to a random walk conditioned on the event $\{\tau_\vartheta {>}n\}$ when $n\to\infty$, i.e.   to stay in $E$: if the sequence of ratios $(\P_x(\tau_\vartheta {>} n)/ \P_e(\tau_\vartheta{>}n))$ converges to some harmonic function $(h(x))$, the corresponding conditioned random walk is the $h$-transform of the original killed random walk. Random walks conditioned to stay in cones have been investigated in a large number of references, see for instance  Bertoin and Doney~\cite{Bertoin-Doney} and  Denisov and Wachtel~\cite{Denisov-Wachtel} and the references therein.  

For one{-}dimensional random walks with a non-negative drift, killed up on  the time of the first exit from a half line, the harmonic functions and the corresponding random walk conditioned to stay in the half line can be represented in terms of the renewal function based on the ladder heights by using  the Wiener-Hopf factorization.  The method of the Wiener-Hopf factorization is a quite powerful  technique. It   does not require  any moment conditions  in the one{-}dimensional case, see for instance Spitzer~\cite{Spitzer}, Feller~\cite{Feller} and  Bertoin and Doney~\cite{ Bertoin-Doney}.  For multidimensional random walks,  a similar  approach seems to be unlikely because there is, up to now,  no equivalent of  a convenient Wiener-Hopf factorization. There are nevertheless some results in this domain,  Mogulskii and Pecherskii in~\cite{Mogulskii-Pecherskii} and Greenwood and Shaked~\cite{Greenwood:1, Greenwood:2} where some factorization relations are established  to get the asymptotic behavior of some  first passage times.  It is not clear however how these factorization relations can be used to investigate the harmonic functions and  the corresponding random walk conditioned to stay in a cone. 

We review briefly several approaches used to identify the set of  harmonic functions and to investigate the limit of the sequence $
(\P_x(\tau_\vartheta {>} n)/ \P_e(\tau_\vartheta{>}n))$ 
for multidimensional random walks.

To identify the Martin boundary for integer-valued two-dimensional random walks  when the size of the jumps of each coordinate is at most $1$,  complex analysis methods on elliptic curves has been proposed by Kurkova and Malyshev~\cite{Kurkova-Malyshev},  Kurkova and Raschel~\cite{Kurkova-Raschel} and Raschel~\cite{Raschel:1}.  This approach does not seem to have an extension in higher dimensions or for random walks with unbounded jumps.  

A large deviations approach has been proposed in Ignatiouk and Loree~\cite{Ignatiouk-Loree}  to identify the Martin boundary of two{-}dimensional, non centered, random walks killed up on  the first exit from the positive quadrant $\N\times\N$. For centered random walks, such a method does not seem to work. Moreover, to apply such a method for a non centered random walk in $\Z_+^d$,  one has to identify the harmonic functions for centered random walks in $\Z_+^{k}$ for any $k {<} d$, see Ignatiouk~\cite{Ignatiouk:08}.

Recently,  Bouaziz et al.~\cite{Sami} proved the uniqueness of the positive harmonic function, up to a multiplicative constant,  for a centered random walk killed up on the time of the first exit from $\Z_+^d$. The result has been obtained  under the following conditions: the jumps of the random walk must be bounded and their distribution has to  satisfy some ellipticity condition. 
A large number of related results were obtained for random walks in Weyl chambers, see for instance  Gessel and Zeilberger~\cite{Gessel}, Grabiner and Magyar~\cite{Grabiner}, Koenig and Schmid~\cite{Koenig} and Raschel~\cite{Raschel:2, Raschel:3}, and the references therein.  These results  concern some specific  random walks, when the steps of the process cannot jump over the boundary of the chamber and  when the set of all possible steps is invariant with respect to some reflections.

For centered random walks killed up on the time of the first exit from some general cone, under some moment condition on the jumps,  a harmonic function has been constructed and the asymptotic of the tail distribution of the exit time has been obtained by Denisov and Wachtel~\cite{Denisov-Wachtel}, and in a recent paper of Raschel and Tarrago \cite{Raschel:4} the asymptotic behavior of the Green function was obtained and a uniqueness (up to a multiplicative constant) of a positive harmonic function was proved. For non{-}centered random walks satisfying a Cramer condition the results of  Denisov and Wachtel~\cite{Denisov-Wachtel} have been  extended by Duraj~\cite{Duraj}. These results use the diffusion approximation of random walks and seem to be impossible to extend for random walks with heavy tailed jumps.

The purpose of the present paper is to extend the notions of  the ladder heights and of the corresponding renewal function to a multi-dimensional setting, and to obtain  results similar to those in the classical fluctuation theory for one dimensional random walks without any assumption on the tail distributions of the jumps.  For this we combine the large deviation estimates  with a modified method  of Choquet-Deny theory. 

We consider a substochastic transient random walk on a countable semi-group $(E,{\star})$ with  identity element $e$. The transition probabilities of our random walk are assumed to satisfy the inequalities $p(x{\star} u, y{\star} u) {\geq} p(x,y)$ for all $x$, $y$, $u{\in}E$. A homogeneous random walk on $\Z^d$ killed up on the time of the first exit from some convex cone with a vertex at the origin $0$ is therefore a particular case. 

To define the ladder height process and the corresponding renewal function $V$,  an additional absorbing state $\vartheta$ is added so that our random walk $(X(t))$ has a stochastic kernel,   $\tau_\vartheta$ is defined as the hitting time of  $\vartheta$. For  a random walk on $\Z^d$ killed up on the time of the first exit from the cone, this additional state $\vartheta$ is in fact identified with the points of $\Z^d$ outside of the cone and $\tau_\vartheta$ represents the time of the first exit from the cone. When the transition probabilities of our random walk satisfy the identities$p(x{\star} u, y{\star} u) {=} p(x,y)$ for all $x$, $y$, $u{\in}E$, the ladder height process $(H_n)$ has a simple probabilistic interpretation: for a sequence of stopping times $(t_n)$ with $t_0{=} 0$ and  $t_n$ defined as the first time after $t_{n-1}$ when the random walk $(X(t))$ exits from the set $E{\star}X(t_{n-1})$, in distribution $H_n{=}X(t_n)$ if $t_n {<}{+}\infty$ and $H_n{=}\vartheta$ otherwise. In the general case, the ladder height process is defined by positive operators represented in terms of the differences $p(x{\star} u, y {\star} u){-} p(x,y)$.

For a random walk in $\Z$, killed up on  the first exit from $\Z_+$, i.e. when $({\cal E},{\star}){=} (\Z, {+})$, our definition of the ladder height process $(H_n)$ coincides with the classical definition. However, in contrast to the ladder height process for a one dimensional random walk, for a multidimensional random walk, the ladder height process is no longer  decreasing. This is the main technical difficulty in our approach. 

The renewal function $V$ corresponding to the ladder height process is defined then by $V(x){=}\E_x({\cal T}_\vartheta)$ where ${\cal T}_\vartheta$ denotes the first time when the process $(H_n)$ hits the absorbing state $\vartheta$. When the stopping time $\tau_\vartheta$ is not integrable  we prove that the renewal function $V$  is harmonic for the random walk $(X(t))$, and that the sequence of functions 
\be\label{eq0-00}
\P_x(\tau_\vartheta > n)/\P_e(\tau_\vartheta > 0),  \quad x\in E. 
\ee
converges to $V$. 
When the  stopping time $\tau_\vartheta$  is integrable, an analogue of the Wald identity 
\[
\E_x(\tau_\vartheta) = V(x) \E_e(\tau_\vartheta), \quad \forall x\in E, 
\]
is obtained, and we show that 
\[
V(x) ~\leq~ \liminf_{n\to\infty} \P_x(\tau_\vartheta > n)/\P_e(\tau_\vartheta > n).
\]
The harmonic functions are investigated in a particular case, when the hitting probabilities of the random walk $(X(t))$ are slowly varying, i.e. when for any $u\in E$ and uniformly on $x\in E$,
\begin{align*}
\lim_{n\to\infty} \frac{1}{n} &\log \P_x(X(t) = x \star u^{\star n}, \; \text{for some $t\geq 1$}) = \\
&\quad \lim_{n\to\infty} \frac{1}{n} \log \P_{x\star u^{\star n}} (X(t) = x, \; \text{for some $t\geq 1$}) ~=~ 0,
\end{align*} 
by using a modified method of Choquet-Deny theory. 

In the classical homogeneous setting the main idea of the Choquet-Deny theory is the following: to identify the set of  harmonic functions one has to determine the {\em minimal} harmonic functions. A harmonic function $h$ is minimal if, for any  other harmonic function $\tilde{h}$, the inequality $\tilde{h}{\leq}h$ implies that $\tilde{h}{=}ch$ for some constant $c$. If a harmonic function $h$ is minimal for an irreducible homogeneous random walk on  a group $({\cal E},\star)$, then for any $u{\in}{\cal E}$, the function $\tilde{h}(x){=} h(x{\star} u)$ is harmonic and by Harnack's inequality,  satisfies the inequality $\delta_u\tilde{h}{\leq}h$ for some constant $\delta_u{>}0$. Hence, for any $x{\in} E$, $h(x{\star}u){=}\gamma_uh(x)$ for some constant $\gamma_u{>} 0$. The last relation gives a characterization of all minimal harmonic functions. For a homogeneous random walk in $\Z^d$ these are exponential functions.

In our setting,  Choquet-Deny method is modified in the following way:  for  any minimal harmonic function $h$ and  $u{\in}E$, we show that the function $x{\mapsto}h(x{\star}u)$ is super-harmonic  and its harmonic component,  in the Riesz decomposition, is given by  $\gamma_u h$ for some constant $\gamma_u{>}0$. Assuming that the hitting probabilities of the random walk are slowly varying, we prove that $\gamma_u{=}1$. In this way, we get some functional relations proving that any harmonic function $h$ of $(X(t))$ is superharmonic for the ladder height process  $(H_n)$ with a  potential, in the Riesz decomposition, component given by  $h(e)V$ and a harmonic component $\tilde{h}$ given by 
\[
\tilde{h}(x) = \lim_{n\to\infty} \E_x(h(H_n), \; {\cal T}_\vartheta >n), \quad x\in E.  
\]
Using this result together with the results obtained for the renewal function $V$, we conclude that any harmonic function $h$ of $(X(t))$ is 
\begin{itemize}
\item[--]  proportional to the renewal function $V$ if the exit time $\tau_\vartheta$ is not integrable;
\item[--] bounded below by $h(e) V$, so that the difference $h{-}h(0)V$ is a non trivial harmonic function of the ladder height process $(H_n)$, otherwise. 
\end{itemize}  

Our  results are applied next for a random walk $(X(t))$ on $\Z^d$ killed up on the time of the first exit from the closure of some open convex cone having a vertex at $0$. We show that for  a centered irreducible random walk, or more generally, if generating function of the distribution of the jump of the random walk achieves its strict minimum  at the origin $0\in\R^d$, the hitting probabilities are then slowly varying. This result is obtained by using large deviation estimates for truncated random walks. 

The functional relations obtained for the harmonic functions provide a probabilistic interpretation of the values $h(x)/h(x{\star}u)$ in terms of the corresponding $h${-}transform: if the transition probabilities of the random walk $(X(t))$  satisfy the identity $p(x{\star}u, y{\star}u) {=} p(x,y)$ for all $x$, $y$, $u{\in}E$ and the hitting probabilities are slowly varying, then for any harmonic function $h$ and $x$, $u{\in}E$, the probability that the $h${-}transform process starting at $x{\star}u$ never exits from the set $E{\star}u$ is equal to $h(x)/h(x{\star}u)$. This  property could be of interest in view of applications to random walks in $\Z^d$ conditioned to stay in some convex cone. Denisov and Wachtel~\cite{Denisov-Wachtel} have proved that, for a centered random walk in a general cone, under some moment conditions on the distribution of the jumps, the sequence of ratios~\eqref{eq0-00} converges  to some harmonic function ${\cal V}$. The corresponding random walk conditioned on the sequence of events $(\{\tau_\vartheta{>}n\})$ to stay in the cone, is in this case the ${\cal V}$-transform of the original random walk $(X(t))$. 

Remark finally that,  in higher dimensions, a new phenomenon occurs: in contrast to a centered one dimensional random walk, the exit time $\tau_\vartheta$ can be integrable and, in this case, the limiting harmonic function $${\cal V}{=} \lim_n\P_\cdot(\tau_\vartheta{>} n)/\P_0(\tau_\vartheta{>} n)$$ obtained by Denisov and Wachtel~\cite{Denisov-Wachtel} is not equal to the renewal function $V$.

\section{Preliminaries.} 
Before formulating our  results we recall well known results in fluctuation theory for one dimensional random  walks. The classical references are here the books of Spitzer~\cite{Spitzer} and Feller~\cite{Feller}. 

Denote by $(S(t))$  a homogeneous irreducible and aperiodic random walk  on $\Z$ and let $(X(t))$ be a copy of $(S(t))$ killed up on the first exit from $\Z_+ = \{k\in\Z: k \geq 0\}$. The strict descending ladder height process $(H_k,t_k)$ for the random walk $(S(t))$ is defined by   
\be\label{eq1-0}
t_0 = 0, \quad t_{k+1} = \inf\{n > t_k :  S(n) <  S(t_k) \}  \quad \text{and} \quad H_k =   S(t_k), \quad k\geq 0.
\ee
$(H_k)$ is therefore a substochastic random walk on $\Z$ with transition probabilities $p_H(x,y) = \P_0( S(\tau) = y-x)$ where $\tau =\inf\{ t > 0 : S(t) < 0\}$. The random walk $(H_k)$ is stochastic if $\P_0$-almost surely $\tau < \infty$. If  $\P_0(\tau = \infty) > 0$ it is convenient to introduce an absorbing state $\vartheta$ for $(H_k)$ by letting 
\[
p_H(x,\vartheta) ~=~ 1 - \sum_{y\in \Z_+} p_H(x,y), \quad x\in \Z_+.
\]
The renewal function $V$ associated with $(H_k)$ is then defined by  
\[
V(x) ~=~ \sum_{k=0}^\infty \P_x(H_k \in \Z_+), \quad x\in\Z_+
\]
or equivalently, 
\[
V(x) ~=~ \E_x({\cal T}), \quad x\in\Z_+.
\]
where $\cal{T}$ denotes the first time when the Markov chain $H_n$ exits from the set $\Z_+$: 
\[
\cal{T} = \inf\{n > 0 : H_n \not\in\Z_+\} ~=~\inf\{n>0 : H_n\in \{\vartheta\}\cup \Z_-\}, \quad \Z_- = \Z\setminus\Z_+. 
\]
If $\P_0(\tau = +\infty) > 0$, i.e. when  the random walk $(S(t))$ drifts to $+\infty$, the function $V$ satisfies the equality 
\be\label{e0-1}
V(x)  ~=~  \P_x(\tau = +\infty)/\P_0(\tau = +\infty), \quad \forall x\in\Z_+.
\ee
If the reflected random walk $(-S(t))$ drifts to $+\infty$, i.e. when $(S(t))$ drifts to $-\infty$, the function $x\to \E_x(\tau)$ is finite everywhere on $\Z$ and 
\be\label{e0-2}
V(x) ~=~ {\E_x(\tau)}/{\E_0(\tau)}, \quad \forall x\in\Z_+.
\ee
Finally, the random walk $(S(t))$ is called oscillating if it  neither drifts to $+\infty$ nor to $+\infty$. In this case, $\E_x(\tau) = +\infty$ and  
\be\label{e0-3}
\lim_{n\to\infty} ~\P_x(\tau > n)/\P_0(\tau>n) = V(x), \quad \forall x\in\Z_+.
\ee
Moreover, if a random walk $(S(t))$ either drifts to $+\infty$ or oscillates, the function $V$ is harmonic  for the killed random walk $(X(t))$ and any harmonic function of $(X(t))$ is proportional to $V$. 

The purpose of the present paper is to extend the notion of the  ladder height process and to obtain similar results for  random walks in semi-groups and in particular, for centered random walks in $\Z^d$  killed up on the first exit from some convex cone. 

\section{Main results} 
We begin our analysis with an application of the Choquet-Deny theory for  a  substochastic Markov chain $(X(t))$ on a countable set $E$  with transition probabilities $p(x,y)$, $x,y\in E$. The  Markov chain $(X(t))$ is assumed to  satisfy  following conditions:

\medskip
{\em 
\begin{enumerate}
\item[(A0)] $(X(t))$ is irreducible and transient on $E$,
\item[(A1)]  the  set of states $E$ of $(X(t))$ is included to some group $({\cal E}, \star)$;  
\item[(A2)]  there is $u\in E$ such that 
\[
E \star u\subset E,
\] 
and for any $x,y\in E$, 
\[
p(x\star u, y \star u) ~\geq~ p(x,y);   
\]
\end{enumerate}
} 
\medskip

The Markov chain $(X(t))$ being substochastic, it is convenient to introduce an additional absorbing state $\vartheta$ by letting
\[
p(x, \vartheta) = 1 - \sum_{y\in E} p(x,y), \quad x\in E. 
\]
$(X(t))$  is then a Markov chain on $E \cup\{\vartheta\}$ stopped when hitting the state $\vartheta$. We denote 
\[
\tau_\vartheta ~=~\inf\{t \geq 1 : X(t) =\vartheta \}. 
\] 
The Green function $G(x,y)$ of the Markov chain  $(X(t))$ is  defined by 
\[
G(x,y) ~=~ \sum_{t=0}^\infty \P_x( X(t) = y) ~=~ \sum_{t=0}^\infty \P_x( X(t) = y, \; t < \tau_\vartheta), \quad x,y\in E,
\]
and the hitting probabilities are denoted by 
\[
Q(x,y) ~=~ \P_{x}\bigl( X(t) = y \; \text{ for some } \; 0 < t < +\infty \bigr), \quad x,y\in E.
\]
For  a non-negative function $\varphi : E \to \R_+$, we let
\[
G\varphi(x) ~=~ \sum_{y\in E} G(x,y)\varphi(y), \quad x\in E, 
\]
and 
\[
P\varphi(x) ~=~ \E_x(\varphi(X(1))) ~=~ \sum_{y\in E} p(x,y) \varphi(y), \quad x\in E,
\]
It is convenient moreover to introduce two operators $\varphi \to T_u\varphi$ and $\varphi \to A_u\varphi$ on the set of non-negative functions $\{\varphi : E \to \R_+\}$, by letting 
\[
T_{u}\varphi (x) ~=~ \varphi(x \star u), \quad x\in E, 
\]
and 
\be\label{e1-2p} 
A_{u} \varphi(x) ~=~  \sum_{y\in E} a_{u}(x,y)\varphi(y), \quad x\in E, 
\ee
with  
\[
a_{u}(x,y) = \begin{cases}p(x\star u, y) - p(x,y\star u^{-1}), &\text{if $y \in E \star u$,} \\
p(x\star u, y) &\text{otherwise.} 
\end{cases} 
\]
Remark that under the hypotheses (A1) and (A2),  $a_{u}(x,y) \geq 0$ for all $x,y\in E$. For  any non-negative function $\varphi : E \to \R_+$, the function $A_{u}\varphi : E \to \R_+\cup\{+\infty\}$ is therefore well defined and  
\[
T_u P\varphi  =  PT_u\varphi + A_u \varphi. 
\]
Recall that for a  Markov chain $(X(t))$, a  non-zero positive function $h: E \to \R_+$ is called harmonic (respectively super harmonic) if $
P h ~=~ h$ (respectively $Ph \leq h$). A  function $g : E \to \R_+$ is called potential for $(X(t))$ if $g = G\varphi$ with some non-negative function $\varphi : E \to \R_+$. Such a  function $\varphi : E \to \R_+$ satisfies then the following relation 
\[
\varphi =  g - P g.
\]
Any potential function is super harmonic, and by the Riesz decomposition theorem, any super harmonic function $f$ is equal to a sum of a harmonic function $h = \lim_n P^n f$ and a potential  function $g = G\varphi$ with $\varphi = (\Id-P) f$, see for instance Woess~\cite{Woess}. 

We  extend any harmonic or super harmonic function $h$ on $E\cup\{\vartheta\}$ by letting 
\[
h(\vartheta) = 0. 
\]
$\1$ denotes throughout this paper the identity constant function, $\1(x) = 1$ for all $x\in E$, and $\Id$ denotes the identity operator : $\Id\varphi = \varphi$ for any function $\varphi:E \to \R$. 

\begin{defi} 
We will say that the hitting probabilities of the random walk $(X(t))$ are slowly varying along the element $u\in E$ if uniformly on $x\in E$, 
\be\label{equation1-2}
\lim_{n\to\infty} \displaystyle{\frac{1}{n}} \log Q(x, x \star u^{\star n}) ~=~ \lim_{n\to\infty} \displaystyle{\frac{1}{n}} \log Q(x \star u^{\star n}, x) = 0,
\ee 
where $u^{\star 1} = u$ and $u^{\star(n+1)} = u^{\star n}\star u$ for $n\geq 1$.
\end{defi} 
Our first result is the following statement.

\medskip 
\begin{theorem}\label{th1} Suppose that (A0)-(A2) are satisfied and let for the given $u\in E$, the hitting probabilities of $(X(t))$ be slowly varying along $u$. 
Then  any harmonic function $h$ satisfies the  following relations 
\be\label{e1-2}
h(y\star u) ~=~ h(y) + G A_u h(y), \quad \forall y\in E. 
\ee
\end{theorem} 
The proof of this result uses the arguments of Choquet-Deny theory and is given in Section~\ref{sec2}. In Section~\ref{applications} we apply our results for homogeneous random walks in $\Z^d$ killed up on  the time of the first exit from some general cone. The results of this section show that under quite general assumptions on the cone, centered random walks satisfy \eqref{equation1-2}. 
\medskip

The functional relations \eqref{e1-2} are the key point of our approach. Before formulating our next results, we rewrite  these relations for the case when (A2) holds with the equality and we obtain a probabilistic interpretation of  \eqref{e1-2} in terms of the corresponding $h$-transform. 
For this  it is convenient to introduce the first time when the process $(X(t))$ exits from $E\star u$ : 
\[
\eta_u = \inf\{t \geq 1:~X(t)\not\in E\star u\}.
\]
Recall that for a given non-zero harmonic function $h$ of $(X(t))$, the $h$-transform $(X_h(t))$ of the process $(X(t))$ is defined as a 
Markov chain with transition probabilities 
\[
p_h(x,y) ~=~ p(x,y) h(y)/h(x), \, x,y \in E.
\]
For any non-zero harmonic function $h$ of $(X(t))$, the transition matrix of the $h$-transform $(X_h(t))$  is stochastic on $E$ and consequently, the process $(X_h(t))$ does not exit from the set $E$. If  the sequence of functions 
\be\label{eq1-50}
f_n(x) = \P_x(\tau_\vartheta > n)/\P_e(\tau_\vartheta > n),  \quad x\in E
\ee
converges   to some non zero harmonic function $h: E\to \R_+$, then the random walk $(X(t))$ conditioned on the sequence of events $(\{\tau_\vartheta > n\})$ to stay in $E$ is the corresponding $h$-transform of $(X(t))$, see for instance Bertoin and Doney~\cite{Bertoin-Doney}.  

\begin{prop}\label{pr-intuition} Suppose that (A0) and (A1) are satisfied and let for some $u\in E$, $E\star u \subset E$ and 
\[
p(x\star u, y \star u) = p(x,y), \quad \forall x,y\in E.
\]
Then  for any non negative function $h: E\to \R_+$, 
\be\label{e1-200}
GA_{u}h (y)  = \E_{y\star u} \bigl(h(X(\eta_u)), \; \eta_u < \tau_\vartheta \leq +\infty \bigr) , \quad \forall y\in E.
\ee
If moreover, the hitting probabilities of $(X(t))$ are slowly varying along the given element $u\in E$,  then for any $x\in E$ and any non zero harmonic function $h$ of $(X(t))$, the quantity $h(x)/h(x\star u)$ is equal to the probability that the corresponding $h$ transform $(X_h(t))$ starting at $x\star u$ never exits from the set $E\star u$. 
\end{prop} 
\begin{proof} Indeed, in this case, for any $y,z\in E$,
\[
G(y,z) ~=~\sum_{t=0}^\infty \P_{y\star u}(X(t)= z\star u, \; \eta_u > t) 
\]
and according to the definition of the matrix $A_u$,
\[
a_{u}(x,y) = \begin{cases} 0, &\text{if $y \in E \star u$,} \\
p(x\star u, y) &\text{otherwise.} 
\end{cases} 
\]
Hence, 
\begin{align*}
GA_{u}h (y)  &=~ \sum_{z,z' \in E} G(y,z)a_u(z,z')h(z') \\
&=~   \sum_{z\in E} \, \sum_{z'\in E\setminus(E\star u)} \, \sum_{t=0}^\infty \P_{y\star u}(X(t) =  z\star u, \;  t < \eta_u) p(z\star u, z')h(z')  \\
&=~ \E_{y\star u}\bigl(h(X(\eta_u)),\, \eta_u < \tau_\vartheta \leq +\infty\bigr).
\end{align*} 
The first assertion of Proposition~\ref{pr-intuition} is therefore proved. The second assertion is a consequence of Theorem~\ref{th1}. To get this assertion, it is sufficient to notice that by \eqref{e1-2} and \eqref{e1-200}, the probability that the $h$-transform $(X_h(t))$ starting at $x\star u$ ever exits from the set $E\star u$ is equal to 
\[
\frac{1}{h(x\star u)} \E_{x\star u} (h(X(\eta_u)); \, \eta_u < \tau_\vartheta \leq +\infty) ~=~  1 - h(x)/h(x\star u). 
\] 
\end{proof}

Now we  introduce an analogue of the ladder height process $(H_n)$ and extend the result of the classical fluctuation theory formulated in the previous section. From now on the random walk $(X(t))$ is assumed to satisfy  the  following conditions:
\medskip
{\em 
\begin{enumerate}
\item[(B1)]  the state set $E$ of $(X(t))$ is included to some group $({\cal E}, \star)$ and $(E,\star)$ is a semi-group with an identity element $e$; 
\item[(B2)]  for any $x,y, u\in E$, 
\[
p\bigl(x \star u, y \star u\bigr) ~\geq~ p(x,y),  
\]
\end{enumerate} 
} 

\medskip
Remark that under the above assumptions, for any $u\in E$, the previous conditions (A1) and (A2)  are satisfied  and consequently, the operator $A_u$  on the set of non-negative functions $\{\varphi : E \to \R_+\}$ is well defined. 

To introduce the ladder Markov chain $(H_n)$ we need the following lemma. 
\begin{lemma}\label{lemma1-1}  Under the hypotheses (A0), (B1) and (B2), the matrix $$P_H = \left(p_H(x,y), \, x,y\in E\right)$$ with 
\be\label{eq2-3}
p_H(x,y)  =  GA_{x}\1_{\{y\}}(e) ~=~ \sum_{z\in E} G(e,z)a_x(z,y), \quad x,y\in E,  
\ee
is substochastic.
\end{lemma}
\begin{proof} Indeed,  the coefficients of the matrix $P_H$ are non-negative  and for any $x\in E$ 
\begin{align*} 
P_H\1(x) &=~ GA_x\1(e) = \lim_{n\to\infty} \sum_{k=0}^n P^kA_x\1 (e) =  \lim_{n\to\infty} \sum_{k=0}^n P^k(T_xP - PT_x)\1 (e) \\
&=~ \lim_{n\to\infty} \left( \sum_{k=0}^n  P^kT_xP\1 (e)  - \sum_{k=1}^{n+1} P^{k}T_x\1 (e) \right) 
\end{align*} 
The above relation combined with the equality $T_x\1 = \1$ implies that  
\begin{align*} 
P_H\1(x) &=~ \lim_{n\to\infty} \left( \sum_{k=0}^n  P^kT_xP\1 (e)  - \sum_{k=1}^{n+1} P^{k}\1 (e) \right) \\
&=~   T_xP\1(e)  - \lim_{n\to\infty} \left( \sum_{k=1}^n  P^k(Id - T_xP) \1 (e)  +  P^{n+1}\1 (e) \right). 
\end{align*} 
Since $P^{n+1}\1 (e) \geq 0$ and 
\[
 P^k(Id - T_xP) \1 (e) ~=~ \sum_{y\in E} \P_e(X(k) = y)\left( 1 - \sum_{z\in E} p(y\star x,z)\right) ~\geq~ 0,
\]
from the last relation it follows  that  $
P_H\1(x) \leq T_xP\1(e) ~\leq~1$.  
\end{proof} 

\begin{defi}The ladder height process relative to the Markov chain $(X(t))$ is  defined as a  discrete time Markov chain $(H_n)$ on $ E$ with transition matrix $P_H$. 
\end{defi} 
The  Markov chain $(H_n)$ being sub-stochastic on $E$, we introduce an additional absorbing state $\vartheta$. Without any restriction of generality we can keep the same notation of the additional absorbing state as for the random walk $(X(t))$.   We let 
\be\label{eq2-4p}
p_H(x,\vartheta) = 1 - \sum_{y\in E} p_H(x,y), \quad \forall x\in E.
\ee
Before formulating our next result let us give another equivalent definition of the  ladder Markov chain $(H_n)$ in a  particular case, when (B2) holds with the equality. The following statement is a straightforward consequence of Proposition~\ref{pr-intuition} 

\begin{prop}\label{prop1-1} Suppose that (A0) and (B1) are satisfied and let 
\[
p(x \star u, y \star u) = p(x,y), \quad \forall u,x,y\in E. 
\]
Then for a sequence of random times $(t_n)$ defined for $n\in\N$ by 
\be\label{eq1-1001} 
t_0 = 0, \quad \text{and} \quad  t_{k+1} = \begin{cases}\inf\{ n > t_k : X(n) \not\in E \star X(t_k)\}, &\text{if $t_k <+\infty$},\\ +\infty &\text{otherwise,}
\end{cases}
\ee
in distribution 
\be\label{e1-10}
H_k = \begin{cases} X(t_k) &\text{if $t_k < \infty$}\\
\vartheta, &\text{otherwise}.
\end{cases} 
\ee 
\end{prop} 

\medskip 

Now we extend the notion of the renewal function $V: E\cup\{\vartheta\}\to\R_+\cup\{+\infty\}$ . Denote  ${\cal T}_\vartheta = \inf\{ n \geq 0 : H_n = \vartheta\}$ and let 
\be\label{eq-V}
V(\vartheta) =  0 \quad \text{and} \quad V(x) ~=~  \E_x({\cal T}_\vartheta), \quad x \in E. 
\ee
Our next result proves that this function is finite and satisfies the properties similar to those of the renewal function for a one-dimensional random walk. Recall  that under (A0), the Markov chain $(X(t))$ is irreducible in $ E$. The  function $x \to \E_x(\tau_\vartheta)$ is  therefore either finite in $ E$, or infinite everywhere in  $ E$. 
\begin{theorem}\label{th2} 
Under the hypotheses (A0), (B1) and (B2) the following assertions hold :
\begin{enumerate} 
\item[(i)] The ladder height process $(H_n)$ is transient on $E$. 
\item[(ii)] The function $V$ is finite and satisfies the following relations ~:
\be\label{eq2-5b}
V(e) ~=~1
\ee
and 
\be\label{eq2-5a} 
V(x \star u) ~=~ V(x) + GA_uV(x), \quad \forall x,u\in E. 
\ee
\item[(iii)] If $\E_\cdot(\tau_\vartheta) = +\infty$, the function $V$ is  harmonic for the Markov chain $(X(t))$ and for  any $x\in E$,
\be\label{eq2-5}
\lim_{n\to\infty} {\P_x(\tau_\vartheta > n)}/{\P_e(\tau_\vartheta > n)} ~=~ V(x).  
\ee
\item[(iv)] If $\E_\cdot(\tau_\vartheta) < +\infty$,  the function $V$ is  potential for the Markov chain $(X(t))$ and  for  any $x\in E$,  
\be\label{eq2-6}
V(x) ~=~ {\E_x(\tau_\vartheta)}/{\E_e(\tau_\vartheta)}  ~\leq~ \liminf_{n\to\infty} ~{\P_x(\tau_\vartheta > n)}/{\P_e(\tau_\vartheta > n)}.  
\ee
\end{enumerate} 
\end{theorem} 
The proof of this theorem is given in section~\ref{proof-th2}. 

Remark  that under the hypotheses of Theorem~\ref{th1} and Theorem~\ref{th2}, by \eqref{eq2-5a} and \eqref{e1-2}, 
the renewal function $V$ satisfies the same functional relations as any harmonic function $h$. This is the key point of the proof of our next result.

From now on the Green's function of the ladder height process $(H_n)$ will be denoted by 
\[
G_H(x,y) ~=~ \sum_{n=0}^\infty \P_x(H_n = y), \quad x,y\in E.
\]
For $n\geq 1$ and  a non negative function $\varphi :  E\to\R_+$ we introduce 
\[
P_H^n\varphi(x) ~=~ \E_x(\varphi(H_n); \, {\cal T}_\vartheta > n) 
\]
and 
\[
G_H\varphi(x) ~=~ \sum_{n=0}^\infty P_H^n\varphi(x) ~=~ \sum_{n=0}^\infty  \E_x(\varphi(H_n); \, {\cal T}_\vartheta > n).
\]

\begin{theorem}\label{th3} Suppose that the conditions (A0), (B1) and (B2) are satisfied and let  the hitting probabilities of the random walk $(X(t))$ be slowly varying along every  $u\in E$. 
Then  any harmonic function $h$ of  $(X(t))$ is super harmonic for the ladder height process $(H_n)$ and 
\be\label{eq2-7n}
h(x) ~=~h(e) V(x) + \tilde{h}(x), \quad \forall x\in E,  
\ee
with a harmonic, for the ladder height process $(H_n)$, function 
\be\label{eq2-8n}
\tilde{h}(x) = \lim_{n\to\infty} \E_x(h(H_n); \, {\cal T}_\vartheta > n),\quad \forall x\in E. 
\ee
If moreover $\E_\cdot(\tau_\vartheta) = +\infty$, then for any harmonic function $h$, the function \eqref{eq2-8n} is equal to zero and $h$  is proportional to $V$. Otherwise, for any harmonic function $h$, the function \eqref{eq2-8n} is 
non trivial. 
\end{theorem} 
The proof of this theorem is given in section~\ref{proof-th3}.

\section{Application to a homogeneous random walk}\label{applications} 
In this section, we apply our results for a homogeneous random walk in $Z^d$ killed up on the time of the first exit from a convex cone.  

Consider  a probability measure $\mu$ on  $\Z^d$ and let $(X(t))$ be a substochastic  random walk   on $E$   with transition probabilities  $\P_x(X(1)=y) = p(x,y) = \mu(y-x)$, $x,y\in E$. 
Such a random walk is a copy of a homogeneous random walk $(S(t))$ in $\Z^d$ with transition probabilities $p(x,y) = \mu(y-x)$, killed up on the first exit from $ E$. 
The additional absorbing state $\vartheta$ is then identified with the set $\Z^d\setminus E$: for every $x\in E$ we let 
\[
p(x,\vartheta) = 1 - \sum _{y\in E} \mu(y-x) 
\]
We will assume that 
{\em 
\begin{enumerate}
\item[(C0)] the random walk $(X(t))$ is  irreducible  and transient on $E$;
\item[(C1)]   $0 \in E $ and for any $x,u\in E$, \; $x+u\in E$. 
\end{enumerate} 
}

Remark that under the above assumptions, the conditions (A0) , (B1) and (B2) of the previous section are  satisfied and consequently,  the ladder  process $(H_n)$ relative to the random walk $(X(t))$ is well defined. Moreover, in this case, (B2) holds with the equality, and hence, by Proposition~\ref{prop1-1},  for a sequence of random times $(t_n)$ defined  by 
\[
t_0 = 0, \quad \text{and} \quad  t_{k+1} = \begin{cases} \inf\{ n > t_k : X(n) \not\in E + X(t_k)\} &\text{ if $t_k <\infty$,}\\ 
+\infty &\text{otherwise} 
\end{cases} 
\]
in distribution 
\be\label{eq1-60}
H_k = \begin{cases} X(t_k) &\text{if $t_k < \infty$}\\
\vartheta, &\text{otherwise}.
\end{cases} 
\ee 
\medskip
Consider now the stopping times  ${\cal T}_\vartheta = \inf\{ n > 0 : H_n = \vartheta\}$ and $\tau_\vartheta = \inf\{ n > 0 : X(n) = \vartheta\}$, and let for $x\in E$,  $\eta_x = \inf\{ t > 0 : X(t)\not\in E + x\}$. Recall that the renewal function $V:E\cup\{ \vartheta\} \to\R_+$ is defined by 
$V(x) = \E_x({\cal T}_\vartheta)$ with $V(\vartheta)=0$. The following statement is a  consequence of Theorem~\ref{th2}.

\begin{cor}\label{corC-1} Under the hypotheses (C0) and (C1) the following assertions hold. 
\begin{enumerate} 
\item[(i)] The ladder height process $(H_n)$ is transient on $E$. 
\item[(ii)] The function $V$ is finite with $V(0)~=~1$ and satisfies the following relations 
\[
V(x+u) ~=~ V(x) + \E_{x+u}\bigl(V(X(\eta_u)), \, \eta_u < \tau_\vartheta \leq \infty\bigr), \quad \forall \; x,u\in E.
\]
\item[(iii)] If $\E_\cdot(\tau_\vartheta) = +\infty$, the function $V$ is  harmonic for the Markov chain $(X(t))$ and for  any $x\in E$,
\[
\lim_{n\to\infty} {\P_x(\tau_\vartheta > n)}/{\P_0(\tau_\vartheta > n)} ~=~ V(x).  
\]
\item[(iv)] If $\E_\cdot(\tau_\vartheta) < +\infty$,  the function $V$ is  potential for the Markov chain $(X(t))$ and  for  any $x\in E$,  
\[
V(x) ~=~ {\E_x(\tau_\vartheta)}/{\E_0(\tau_\vartheta)}  ~\leq~ \liminf_{n\to\infty} ~{\P_x(\tau_\vartheta > n)}/{\P_0(\tau_\vartheta > n)}.  
\]
\end{enumerate} 
\end{cor} 
\begin{proof} The assertions (i), (iii) and (iv) are proved by the corresponding assertions of Theorem~\ref{th2}. The assertion (ii)  follows from the second assertion of Theorem~\ref{th2} combined with Proposition~\ref{pr-intuition} \end{proof} 
\medskip 
\noindent 
Suppose now that 
{\em 
\begin{enumerate}
\item[(C0')] the random walk $(X(t))$ is transient on $ E$ and satisfies the following communication condition~:  there are $\kappa >0$ and a finite set ${\cal E}_0 \subset \Z^d$ such that 
\begin{itemize}
\item[(a)] $\mu(x) > 0$ for all $x\in{\cal E}_0$; 
\item[(b)] for any $x\not= y$, $x,y\in E$ there exists a sequence $x_0, x_1, \ldots , x_n\in E$ with $x_0=x$, $x_n=y$ and $n\leq \kappa |y-x|$ such that $x_j-x_{j-1}\in {\cal E}_0$ for all $j\in\{1,\ldots,n\}$. 
\end{itemize} 
\item[(C1')]   there is an open convex cone ${\cal C}$  in $\R^d$ with a vertex in the origin $0\in\R^d$, such that $E =  \overline{{\cal C}}\cap \Z^d$, where $\overline{\cal C}$ denotes the closure of ${\cal C}$ in $\R^d$. 
\item[(C2)] the jump generating  function $R:\Z^d \to \R\cup\{+\infty\}$ defined by 
\[
R(\alpha) ~=~ \sum_{x\in\Z^d} \mu(x) e^{\langle \alpha, x\rangle}, \quad \alpha\in\R^d,  
\] 
achieves its minimum at $0\in\R^d$ and for any $\alpha\in\R^d\setminus\{0\}$, $R(\alpha) > R(0) = 1$. 
\end{enumerate} 
}
\noindent
Remark that the last condition is satisfied for integrable and centered random walks, when 
\[
\sum_{x\in\Z^d} x\mu(x) ~=~0,
\]
and also for non integrable random walks if $R(\alpha) = + \infty$ for all $\alpha\not= 0$. 

Our next result proves that under the above assumptions, the hitting probabilities 
\[
Q(z,y) ~=~\P_z(X(t) = y \; \text{for some} \; t < \tau_\vartheta), \quad x,y\in E. 
\]
are slowly varying along any vector $u\in E$. 

\begin{prop}\label{prC-1} Under the hypotheses (C0'), (C1') and (C2),  for any $u\in E$ and uniformly on $x\in E$,   
\be\label{equation1-61}
\lim_{n\to\infty} \displaystyle{\frac{1}{n}} \log Q(x, x + n u) ~=~ \lim_{n\to\infty} \displaystyle{\frac{1}{n}} \log Q(x +nu, x) = 0,
\ee 
\end{prop} 
The proof of this proposition is given in Section~\ref{proof_of_prC-1}. When combined with Theorem~\ref{th1} and Theorem~\ref{th3}  this result provides the following statement.
\begin{cor}\label{corC-2} Under the hypotheses (C0'), (C1') and (C2), any harmonic function $h$ of  $(X(t))$ is super harmonic for the ladder height process $(H_n)$ and satisfies the relation 
\be\label{eq-h}
h(x+u) ~=~ h(x) + \E_{x+u}\bigl(h(X(\eta_u)), \, \eta_u < \infty\bigr), \quad \forall \; x,u\in E,
\ee
and 
\[
h = h(0) V + \tilde{h} 
\]
with a harmonic, for the ladder height process $(H_n)$, function 
\be\label{eq-hp}
\tilde{h}(x) = \lim_{n\to\infty} \E_x(h(X(t_n)), \, t_n < \infty) ~\geq~0.
\ee
If moreover $\E_x(\tau_\vartheta) = +\infty$, then for any harmonic function $h$, the function \eqref{eq-hp} is equal to zero and $h$ is proportional to $V$. Otherwise, the function \eqref{eq-hp} is non trivial. 
\end{cor} 

A straightforward consequence of Corollary~\ref{corC-1} and Corollary~\ref{corC-2} is the following statement. 

\begin{cor} Suppose that the conditions (C0'), (C1') and (C2) are satisfied, the steps of the random walk $(S(t))$ are integrable:
\[
\sum_{x\in \Z^d} |x| \mu(x) < +\infty,
\]
and let 
\[
m ~=~ \sum_{x\in\Z^d} x\mu(x) \not= 0 \quad \text{and} \quad m\in {\cal C}.
\]
Then for any $x\in E$, $V(x) = \P_x(\tau_\vartheta = +\infty)/\P_0(\tau_\vartheta = +\infty)$  and any harmonic function $h$ of the killed random walk $(X(t))$ is proportional to $V$.
\end{cor} 
\begin{proof} If the steps of the random walk $(S(t))$ are integrable, for any $x\in\Z^d$, $\P_x$- almost surely  
\[
\lim_{n\to\infty} S(n)/n = m \in {\cal C}.  
\]
The cone ${\cal C}$ being open, from this it follows that for any $x\in E$, $\P_x$-almost surely 
\[
\sup\{n \in\N~: S(n)\not\in {\cal C}\} < \infty 
\]
and consequently, $\P_x$-almost surely 
\[
\sup\{N :~S(n) + N m \not\in {\cal C}  \quad \text{for some} \quad n\in\N\} ~<~\infty. 
\]
This proves that for any $x\in E$ there is $N \geq 0$ such that 
\[
P_x( S(n) + Nm \in {\cal C}, \; \forall n\geq 0) > 0.
\] 
For $y\in (C + N m )\cap \Z^d$,  one gets therefore 
\begin{align*}
\P_{x+ y} (S(n) \in {\cal C}, \; \forall n\geq 0) &= P_x( S(n) + y \in {\cal C}, \; \forall n\geq 0) \\ & \geq P_x( S(n) + Nm \in {\cal C}, \; \forall n\geq 0) > 0
\end{align*} 
because $C- N m \subset C-y$, and consequently, 
\[
\P_{x+ y}(\tau_\vartheta = +\infty) \geq \P_{x+ y} (S(n) \in {\cal C}, \forall n\geq 0)  > 0
\]
The random walk $(X(t)$ being irreducible on $E = \Z^d\cap E$, from this it follows that the function $x\to \P_x(\tau_\vartheta = +\infty)$ is non zero everywhere on $E$ and in prticular, the exit time $\tau_\vartheta$ is non-integrable. Using Corollary~\ref{corC-1} we conclude therefore that the function $V$ is harmonic for the killed random walk $(X(t))$, 
\[
V(x) ~=~ \lim_{n\to\infty} {\P_x(\tau_\vartheta > n)}/{\P_0(\tau_\vartheta > n)} ~=~ {\P_x(\tau_\vartheta + \infty)}/{\P_0(\tau_\vartheta = \infty)}, 
\]
for any $x\in E$ and by Corollary~\ref{corC-2}, any harmonic function $h$ of $(X(t))$ is proportional to $V$. 
\end{proof} 

Recall that for centered random walks, i.e. when 
\[
\sum_{x\in \Z^d} x \mu(x) ~=~0, 
\]
under some additional moment conditions on the jump distribution $\mu$,  the asymptotic of the tail distribution of $\tau_\vartheta$ were obtained by Denisov and Wachtel~\cite{Denisov-Wachtel}. They construct a harmonic function ${\cal V}$ by using the harmonic function of the corresponding diffusion approximation, and they proved that for any initial position $x$, 
\be
\P_x(\tau_\vartheta > n) ~\sim~ \kappa {\cal V}(x) n^{-p/2} 
\ee
with some  constant $\kappa >0$ and some $p>0$ depending on the cone and also on the covariance matrix of the process. This result provides the pointwise convergence 
\be\label{p-w-onvergence} 
\lim_{n\to\infty}\P_x(\tau_\vartheta > n)/\P_0(\tau_\vartheta > n) = {\cal V}(x), \quad x\in E. 
\ee
Under the hypotheses of the paper ~\cite{Denisov-Wachtel}, the random walk $(X(t))$ conditioned on the events $\{\tau_\vartheta >n\}$ to stay in the cone, is therefore the ${\cal V}$-transform of $(X(t))$, see for instance~\cite{Bertoin-Doney} for more details. Unfortunately, the representation of the function ${\cal V}$ obtained in  ~\cite{Denisov-Wachtel} is rather implicit and the probabilities related to the corresponding conditioned random walk are difficult to investigate. With our approach, under the hypotheses (C0') and (C1')  together with the hypotheses of ~\cite{Denisov-Wachtel}, by Corollary~\ref{corC-1} and~\ref{corC-2}, we obtain for the harmonic function ${\cal V}$ constructed  in ~\cite{Denisov-Wachtel} the following properties~: 
\begin{itemize}
\item[--]  For any $x,u\in E$, the function ${\cal V}$ satisfies the equality
\[
{\cal V}(x+u) ~=~ {\cal V}(x) + \E_{x+u}\bigl({\cal V}(X(\eta_u)), \, \eta_u < \infty\bigr).
\]
\item[--] For any $x,u\in E$, the quantity ${\cal V}(x)/{\cal V}(x+u)$ is equal to the probability that the random walk conditioned on the sequence of events $\{\tau_\vartheta > n\}$ to stay in $E$, starting at $x+u$ never exists from the set $E + u$. 
\item[--] In particular, for any $x\in E$, the quantity $1/{\cal V}(x)$ is equal to the probability that the  conditioned random walk starting at $x$ never exits from $E+x$, 
\item[--] If $p\leq 2$, then the function ${\cal V}$ is identical to the renewal function $V$,
\item[--] If $p >2$, then ${\cal V} \!=\! V + \tilde{h} = \E_\cdot(\tau_\vartheta)/\E_0(\tau_\vartheta) + \tilde{h}$ with a non trivial harmonic for the ladder height process function 
$
\tilde{h} \!=\! \lim_{n\to\infty} \E_\cdot({\cal V}(X(t_n)))$
\end{itemize} 

\medskip
\noindent
{\bf Example. }  Consider a two dimensional irreducible random walk $X(t) = (X_1(t),X_2(t))$ in $\Z^2_+$ killed upon the first exit from the positive quadrant $\Z_+^2$, with transition probabilities $\P_x(X(1)=y) = p(x,y) = \mu(y-x)$, $x,y\in\Z_+^2$ such that 
\[
\sum_{x\in\Z^2} x \mu(x) = 0.
\]
Assume moreover that 
\[
\sum_{x=(x_1,x_2) \in \Z^2} x_1^2 \mu(x) = \sum_{x=(x_1,x_2) \in \Z^2} x_2^2 \mu(x) = 1 \quad \text{and} \quad \sum_{x=(x_1,x_2) \in \Z^2} x_1x_2  \mu(x)  = \rho\in ]-1,1[.
\]
In this case $p= \pi/\arccos(-\rho)$ (see example 3 of the paper ~\cite{Denisov-Wachtel}), and if 
\[
\sum_{x\in\Z^2} \vert x \vert^{p} \mu(x) ~<~+ \infty, 
\]
then 
\[
\P_x(\tau_\vartheta > n) ~\sim~ \kappa {\cal V}(x) n^{-\pi/(2\arccos(-\rho))}, \quad x\in\Z_+^2 
\]
Hence, the exit time $\tau_\vartheta = \inf\{n \geq 0 : X(n)\not\in \Z_+^2\}$ is integrable if and only if $\rho < 0$.  Using our results one gets therefore the following statements : 
\begin{itemize}
\item[--] if $\rho \geq 0$, then,  up to multiplications by constants, the renewal function $V$ is the unique harmonic function for the Markov chain $(X(t))$, and for  any $x\in \Z_+^2$,
\[
V(x) ~=~ \lim_{n\to\infty} {\P_x(\tau_\vartheta > n)}/{\P_0(\tau_\vartheta > n)} ~=~ {\cal V}(x); 
\]
\item[--] if $\rho < 0$, then  the  renewal function $V$ is  potential for the Markov chain $(X(t))$,  
\[
V(x) ~=~ {\E_x(\tau_\vartheta)}/{\E_0(\tau_\vartheta)}  ~  <  ~ \liminf_{n\to\infty} ~{\P_x(\tau_\vartheta > n)}/{\P_0(\tau_\vartheta > n)} ~=~ {\cal V}(x) \quad \forall x\in\Z_+^2
\]
where $\tilde{h} =  {\cal V} - V$ is a non trivial harmonic for ladder height process function.  
\end{itemize} 
Moreover, in a particular case, when $\rho = 0$, the function $h:\Z_+^2\to\R_+$ defined by 
\[
h(x_1,x_2) = (x_1 + 1)(x_2+1), \quad x=(x_1,x_2)\in\Z_+^2,  
\]
is clearly harmonic for $(X(t))$. In this particular case one concludes therefore  that $V = {\cal V} = h$ and for any $x=(x_1,x_2),u=(u_1,u_2)\in\Z_+^2$,  the probability that the random walk conditioned on the sequence of events $\{\tau_\vartheta > n\}$ to stay in $\Z_+^2$, starting at $x+u$ never exists from the set $\Z_+^2 + u$ is equal  to $(x_1+1)(x_2+1)/((x_1+u_1+1)(x_2+u_2+1))$.

\section{Proof of Theorem~\ref{th1} }\label{sec2} 
\subsection{Preliminary results.} We begin the proof of this theorem with the following lemma. It provides Harnack's inequality for super harmonic functions. This first  result is classical and is given here for convenience of the reader to avoid a confusion with another types of Harnack inequalities, see for instance the book of Woess~\cite{Woess}. 

\begin{lemma} \label{lem2-0} For any super harmonic function $f:E\to \R_+$ and $x,y\in E$,
\be\label{eq-Harnack} 
f(x) ~\geq~ Q(x,y)f(y).
\ee
\end{lemma} 
\begin{proof} Let $y\in E$ and denote by $\tau_y$ the first time when the Markov chain $(X(t))$ hits the state $y$: 
\[
\tau_y = \inf\{t \geq 0 :~ X(t) = y\}.
\]
If the function $f: E\to \R_+$ is super harmonic for $(X(t))$ then the sequence $(h(X(n)))$ is a super martingale relative to the natural filtration of $(X(t))$, and by the stopping time theorem, the sequence $(h(X(\min\{n,\tau_y\})))$ is also super martingale. This proves that for  any super harmonic function $h:E\to \R_+$, $n\geq 0$ and $x,y\in E$, 
\[
h(x) ~\geq~ \E(h(X({\min\{n,\tau_y\}}))) ~=~ h(y) \P_x(\tau_y \leq n) + \E(h(X(n)); \,  \tau_y > n) ~\geq~ h(y) \P_x(\tau_y \leq n)
\] 
Letting at the last inequality $n\to+\infty$, by monotone convergence theorem one gets \eqref{eq-Harnack}. 
\end{proof} 

\begin{lemma}\label{lem2-1} Suppose that (A0)  - (A2) are satisfied and let there exists $\delta > 0$ such that 
\be\label{eq30-1}
Q(x, x \star u) ~\geq~\delta \quad \forall x\in E.
\ee
Then the following assertions hold~:
\begin{itemize}
\item[(a)]  for any super harmonic function $f$, the function $T_uf$ is also super harmonic and  satisfies the inequality $
\delta T_u f ~\leq~ f$;  
\item[(b)] for any potential function $g$, the function $T_u g$ is also potential~: 
\[
T_u g = G\varphi \quad \text{with} \quad \varphi~=~ (\Id-P)T_u g.
\] 
\item[(c)] for any  minimal harmonic function $h$, there exists $\gamma \geq 0$ such that for any $n\in\N$, $n\geq 1$,  the function $T_u^n h - \gamma^n h$ is potential and 
\be\label{eq30-2} 
T_u^{n} h - \gamma^n h ~=~ G A_{u^{\star n}} h ~=~ \sum_{k=1}^n \gamma^{k-1} T_u^{n-k} GA_u h 
\ee
\end{itemize}
\end{lemma}
\begin{proof}   Suppose that a function $f: E\to \R_+$ is super harmonic. Then by (A2),
\begin{align*}
P T_u f (x) &=~ \sum_{y\in E} p(x,y) f(y \star u) ~\leq~ \sum_{y\in E} p(x \star u,y\star u)f(y \star u) \\ &\leq~ \sum_{z\in E} p(x\star u, z)f(z) ~\leq~ f(x\star u) ~=~T_uf (x), \quad \forall x\in E. 
\end{align*}
The function $T_u f$ is therefore also super harmonic. Moreover,  from  \eqref{eq30-1} it follows that 
\[ 
\delta T_u f (x) = \delta f(x\star u) ~\leq~  Q(x, x \star u) f(x\star u), \quad \forall x\in E.  
\]
and by \eqref{eq-Harnack}, 
\[
 Q(x, x \star u) f(x\star u) ~\leq~ f(x), \quad \forall x\in E, 
\]
from which it follows  that $\delta T_u f ~\leq~ f$.  The first assertion of our lemma is therefore proved. 

\medskip

To prove the second assertion, recall that every potential function is super harmonic, and a super harmonic function $g$  is potential if and only if $\lim_n P^n g = 0$.  For any potential function $g$, the first assertion of our lemma proves that the function $T_ug$ is super harmonic and 
\[
\lim_n P^nT_u g (x) ~\leq~ \frac{1}{\delta} \lim_n P^n g (x) ~=~ 0, \quad \forall x\in E.
\]
For a potential function $g$, the function $T_u g$ is therefore also potential: $T_ug = G\varphi$ with $\varphi = (\Id-P)T_ug$.

To prove the third assertion of Lemma~\ref{lem2-1} we use the Riesz decomposition theorem. Suppose that  $h$ is a minimal harmonic function. Then from  the first assertion of our lemma  it follows that for any $n\in\N$, the function $T^n_u h = T_{u^{\star n}} h$ is super harmonic and satisfies the inequality $\delta^n T_{u^{\star n}} h ~\leq~ h$. Using the Riesz decomposition theorem,  we conclude that there exist a harmonic function $h_n ~=~\lim_k P^k T_{u^{\star n}} h$ and a potential function $G\varphi_n$ such that 
\[
T_{u^{\star n}} h ~=~ h_n + G\varphi_{n} 
\]
The last relation combined with the inequality  $\delta^n T_{u^{\star n}} h ~\leq~ h$, proves that $\delta^n h_n \leq h$. The harmonic function $h$ being minimal, from this it follows that $h_n ~=~\gamma_{n} h$ for  some $\gamma_{n} \geq 0$, and consequently, 
\be\label{eq30-3}
T_{u^{\star n}} h ~=~ \gamma_{n} h + G \varphi_{n}.
\ee
By  iterating the last equality  with $n=1$, one gets 
\be\label{eq30-4}
T_{u^{\star n}} h ~=~ T^n_u h ~=~ \gamma_1^n h + \sum_{k=1}^n \gamma_1^{k-1} T_u^{n-k} G\varphi_1.
\ee
The second assertion of our lemma proves that  the function 
\[
\sum_{k=1}^n \gamma_1^{k-1} T_u^{n-k} G\varphi_1
\]
is potential  as a sum of potential functions. Hence, by uniqueness of the Riesz decomposition, from \eqref{eq30-3} and \eqref{eq30-4} it follows that 
\[
\gamma_{n} ~=~ \gamma_1^n \quad \text{and} \quad G\varphi_n ~=~ \sum_{k=1}^n \gamma_1^{k-1} T_u^{n-k} G\varphi_1. 
\]
Finally,  recall that  $(\Id-P) h = 0$ and according to the definition of $A_{u^{\star n}}$, 
\[
A_{u^{\star n}}h = T_{u^{\star n}}P h - PT_{u^{\star n}} h.
\]
Hence, from \eqref{eq30-3} one gets  
\[
\varphi_{u^{\star n}} ~=~ (\Id-P) T_{u^{\star n}}h ~=~ T_{u^{\star n}}P h - PT_{u^{\star n}} h ~=~ A_{u^{\star n}} h, 
\]
and consequently, \eqref{eq30-2} holds with $\gamma = \gamma_1$. Lemma~\ref{lem2-1} is therefore proved. 
\end{proof}

\begin{lemma}\label{lem2-2} Under the hypotheses of Theorem~\ref{th1}, there exists $\delta >0$ for which \eqref{eq30-1} holds. 
\end{lemma} 
\begin{proof}
Indeed, from  \eqref{equation1-2}  it follows that there exist $c > 0$ and $n > 0$, $n\in\N$, such that 
\[
Q(x, x \star u^{\star n}) ~>~ c \quad \text{and} \quad  Q(x \star u^{\star (n-1)}, x) ~>~ c, \quad \forall x\in E.
\]
When combined with (A1), these inequalities imply that 
\[
Q(x, x\star u) \geq Q(x, x\star u^{\star n}) Q(x\star u^{\star n}, x \star u) \geq Q(x, x \star u^{\star n}) Q(x\star u^{\star (n-1)}, x) \geq c^2
\]
for all $x\in E$. The last relation proves \eqref{eq30-1} with $\delta = c^2$ 
\end{proof} 

\medskip 

\subsection{Proof of Theorem~\ref{th1}.} 
Now we are ready to prove Theorem~\ref{th1}.  For this it is sufficient to show that under the hypotheses of Theorem~\ref{th1}, every non-zero harmonic function $h$ satisfies  \eqref{eq30-2} with $\gamma=1$. We begin our proof with the case when $h:E\to \R_+$ is a minimal harmonic function for $(X(t))$. 

Suppose that the conditions (A0)-(A2) are satisfied and let for some $u\in E$, \eqref{equation1-2} hold uniformly   on $x\in E$.  Then by  Lemma~\ref{lem2-2}, the inequalities \eqref{eq30-1} hold with some $\delta >0$, and consequently, by  Lemma~\ref{lem2-1}, the function $h$ satisfies the identities \eqref{eq30-2} with some $\gamma >0$. Since $GA_{u^{* n}} h\geq 0$, from \eqref{eq30-2}  it follows that for any $n\in\N$ and $x\in E$, 
\[
T_u^n h (x) \geq \gamma^n h(x).
\] 
Moreover, using \eqref{eq-Harnack} with $y= x \star u^{\star n}$ one gets 
\[
h(x)~\geq~ Q(x, x\star u^{\star n})  h(x \star u^{\star n}) ~=~ Q(x, x\star u^{\star n})  T_u^n h (x),  \quad \forall x\in E
\]
and consequently, 
\[
h(x) ~\geq~ Q(x, x\star u^{\star n}) \gamma^n h(x), \quad \forall x\in E. 
\]
Remark finally that by \eqref{eq-Harnack}, every non-zero harmonic function is strictly positive everywhere on $E$, and consequently, the above inequality implies that 
\[
1 ~\geq~ Q(x, x\star u^{\star n}) \gamma^n, \quad \forall x\in E. 
\]
When combined with  \eqref{equation1-2} the last inequality proves  that 
\[
\log \gamma \leq - \lim_n \frac{1}{n} \log Q(x, x\star u^{\star n}) ~=~ 0
\]
from which it follows that 
\be\label{eq30-5}
\gamma ~\leq~1. 
\ee
To prove that $\gamma \geq 1$,  remark that by \eqref{equation1-2}, for any $\varepsilon > 0$ there exists $n_\varepsilon > 0$ such that 
\[
Q(x \star u^{\star n}, x) ~\geq~ (1-\varepsilon)^n, \quad \forall n\geq n_\varepsilon. 
\]
Hence, for $n\geq n_\varepsilon$, using  again \eqref{eq-Harnack}, one gets 
\[
h(x \star u^{\star n}) ~\geq~ Q(x \star u^{\star n}, x) h(x) ~\geq~ (1-\varepsilon)^n h(x), \quad \forall x\in E.
\]
For $n\geq n_\varepsilon$, the function $T_{u^{\star n}} h - (1-\varepsilon)^n h$ is therefore positive and super  harmonic. Hence, using the same arguments as in the proof of Lemma~\ref{lem2-1}, we obtain 
\[
T_{u^{\star n}} h - (1-\varepsilon)^n h = \gamma_n' h + G \varphi_n'
\]
with some $\gamma'_n \geq 0$ and some potential function $G\varphi'_n$. By uniqueness of the Riesz decomposition, comparison of the above relation with \eqref{eq30-2} proves that for $n\geq n_\varepsilon$, 
\[
\gamma^n = (1 - \varepsilon)^n + \gamma_n' ~\geq~ (1-\varepsilon)^n, 
\]
and consequently, $\gamma \geq 1 - \varepsilon$. Since $\varepsilon > 0$ is arbitrary, from this it follows that $\gamma \geq 1$, and hence using \eqref{eq30-5} we conclude that 
\[
\gamma = 1. 
\]
For a minimal harmonic function $h$, Theorem~\ref{th1} is therefore proved.

\medskip

To extend this result   for an arbitrary harmonic function $h$, we use the Poisson-Martin representation theorem :  for any harmonic function $h$ there exist a positive measure $\nu_h$ on the minimal Martin boundary $\partial_m E$ relative to the Markov chain $(X(t))$ such that 
\be\label{eq30-6}
h(x) ~=~ \int_{\partial_m E} h_\theta(x) \, d\nu_h(\theta), \quad x\in E, 
\ee
where  $h_\theta$ is a minimal harmonic function corresponding to the point $\theta \in \partial_m E$. Since for every $\theta\in\partial_m E$, the following relations  
\[
h_\theta(x\star u) ~=~ T_u h_\theta ~=~ h_\theta ~+~ G A_{u} h_\theta, \quad  u\in E, 
\]
are already proved, from \eqref{eq30-6} by the Fubini-Tonelli  theorem, one gets \eqref{e1-2}. 

\section{Proof of Theorem~\ref{th2}} \label{proof-th2} 
\subsection{Preliminary results.} 

Consider  a sequence of functions $(f_n)$ defined by 
\be\label{eq40-1}
f_n(x) ~=~ \P_x(\tau_\vartheta > n)/\P _e(\tau_\vartheta >n), \quad x\in E. 
\ee
We begin the proof of Theorem~\ref{th2} with the following lemma. 

\begin{lemma}\label{lemma4-1} Suppose that  a Markov chain $(X(t))$ satisfies the conditions (A0), (B1) and (B2).  Then the following assertions hold :
\begin{enumerate}
\item[1)] The sequence of functions $(f_n)$ is relatively compact with respect to the topology of point-wise convergence.
\item[2)]  For any convergent subsequence $f_{n_k}$, the function 
\[
f(x) ~=~\lim_{k\to\infty} f_{n_k}(x), \quad x\in    E, 
\]
is super harmonic for the Markov chain $(X(t))$ and satisfies the inequalities 
\be\label{eq40-2}
  \E_x\bigl(f( H_1), \; {\cal T}_\vartheta > 1\bigr) ~\leq~ f(x) - 1 , \quad \forall x\in    E.
\ee
\item[3)] The function $V(x) = \E_x({\cal T}_\vartheta)$ is finite everywhere on $    E$ and satisfies the inequality 
\[
V(x) ~\leq~ \liminf_{n\to\infty} f_n(x), \quad \forall x\in    E. 
\]
\item[4)] The function $PV$ is finite everywhere on $E$ and satisfies the inequality  $PV(e) \leq 1$. 
\end{enumerate} 
\end{lemma} 
\begin{proof} Indeed, since for any $n\geq 0$ and $x\in    E$,  
\[
  f_n(x) = \P_x(\tau_\vartheta > n)/\P_e(\tau_\vartheta >n) \geq \P_x(\tau_\vartheta > n+1)/\P_e(\tau_\vartheta >n)  =  Pf_n(x) 
\]
the functions $f_n$ are super harmonic for $(X(t))$. Using  \eqref{eq-Harnack} we conclude therefore that  for any $x\in    E$,  
\[
Q(x,e)  ~=~ Q(x,e) f_n (e) ~\leq~f_n(x) ~\leq~ \frac{1}{Q(e,x)} f_n (e) ~=~ \frac{1}{Q(e,x)}, \quad \forall  n\geq 0, 
\] 
and consequently, the sequence of functions $(f_n)$ is relatively compact with respect to the topology of point-wise convergence. 

Consider now a convergent  subsequence $(f_{n_k})$.  Then the function $f = \lim_{k\to\infty} f_{n_k}$ is super harmonic as a limit of super harmonic functions. 
Before proving \eqref{eq40-2} in our general setting let us consider the case when (B2) holds with the equality. Recall that  by Proposition~\ref{prop1-1}, in this particular case, 
in distribution 
\[
H_1 = \begin{cases} X(t_1) &\text{if $t_1 < \infty$}\\
\vartheta, &\text{otherwise}.
\end{cases} 
\]
where $t_{1} = \eta_x ~=~ \inf\{ n > 0 : X(n) \not\in E \star X(0)\}$, and moreover for any $x\in E$ and $n\geq 0$, 
\[
P_x(t_1 > n) = \P_0(\tau_\vartheta > n).
\]
Hence, in the case when (B2) holds with the equality, \eqref{eq40-2}  follows by Fatou lemma from the inequalities 
\begin{align*} 
\P_x(\tau_\vartheta > n) &=  \P_0(\tau_\vartheta > n) + \sum_{y\in E\setminus(E\star x)}\sum_{k=1}^n \P_x(X(t_1) = y, \; t_1 = k) \P_y(\tau_\vartheta > n-k) \\ &\geq \P_0(\tau_\vartheta > n) + \sum_{y\in E\setminus(E\star x)}\sum_{k=1}^n \P_x(X(t_1) = y, \; t_1 = k) \P_y(\tau_\vartheta > n) 
\end{align*} 
To prove \eqref{eq40-2}  in our general setting recall that, according to the definition of $A_x$, $T_xP  =  PT_x + A_x$ and consequently, 
\[
T_xP^{k+1} ~=~ PT_xP^{k} + A_x P^{k}, \quad \forall k\in\N.   
\]
By iterating the last relation on gets 
\[
T_xP^n ~=~ P^nT_x + \sum_{k=1}^n P^{k-1}A_xP^{n-k}, \quad \forall n\geq 1,
\]
from which it follows that 
\[
T_xP^n\1 = P^nT_x\1 + \sum_{k=1}^n P^{k-1}A_xP^{n-k}\1 =  P^n\1 + \sum_{k=1}^n P^{k-1}A_xP^{n-k}\1, \quad \forall n\geq 1.
\]
Since for any $y\in    E$, 
\[
P^{n-k}\1(y) ~=~ \P_y(\tau_\vartheta > n - k) ~\geq~ \P_y(\tau_\vartheta > n) ~=~ P^n\1(y)
\]
and $f_n(y) = P^n\1(y)/P^n\1(e)$, we conclude that 
\[
T_x f_n ~\geq~ \sum_{k=1}^n P^{k-1}A_xf_n + f_n , \quad \forall n\geq 1, \; x\in    E, 
\]
and in particular, since $f_n (e) = 1$, 
\be\label{eq40-3} 
f_n(x) ~=~ T_x f_n (e) ~\geq~ \sum_{k=1}^n P^{k-1}A_xf_n (e) + 1 , \quad \forall n\geq 1, \; x\in    E. 
\ee
Letting therefore $n=n_k \to\infty$ and using Fatou lemma, one gets 
\be\label{eq40-4}
f(x) ~\geq~ \sum_{k=1}^\infty P^{k-1}A_xf (e) + 1 ~=~  GA_xf (e) + 1,  \quad \forall x\in    E.
\ee 
Since, by definition of the ladder process $( H_n)$, 
\[
\E_x(f( H_1), \, {\cal T}_\vartheta > 1) ~=~ GA_xf  (e),
\]
\eqref{eq40-4} proves \eqref{eq40-2}.  The first and the second assertions of our lemma are proved. 

To prove the last assertion, we use again \eqref{eq40-3}. By Fatou lemma,  from \eqref{eq40-3} it follows that for  $f = \liminf_n f_{n}(x)$   the inequalities \eqref{eq40-4}and \eqref{eq40-2} also hold. The iterates of \eqref{eq40-2} show that 
\[
0 ~\leq~ \E_x(f( H_n), \; {\cal T}_\vartheta > n) ~\leq~ f(x) - \sum_{k=0}^{n-1} \P_x({\cal T}_\vartheta > k),  \quad \forall n \geq 0, \, x\in    E. 
\]
Letting therefore $n\to\infty$ one gets 
\[
V(x) ~=~ \E_x({\cal T}_\vartheta) ~\leq~ f(x), \quad \forall x\in    E.
\]
The third assertion of our lemma  is therefore proved. Moreover, the function $f$ being superhamonic, the fourth assertion of Lemma~\ref{lemma4-1} follows from the above inequality in a straightforward way:
\[
PV(x) \leq Pf(x) \leq f(x) < \infty, \quad \forall x\in E
\]
and in particular,
\[
PV(e) \leq f(e) = 1. 
\]
\end{proof} 
As a straightforward consequence of the last assertion of Lemma~\ref{lemma4-1} we obtain 

\begin{lemma}\label{lemma4-1p} Under the hypotheses (A0), (B1) and (B2), for any $x\in E$, the functions $T_xPV$, $PT_xV$ and $A_xV$ are finite everywhere on $E$
\end{lemma} 
\begin{proof} Indeed, by Lemma~\ref{lemma4-1} the function $PV$ is finite on $E$ and hence, for any $x\in E$,  the function $T_xPV$ is also finite. Since according to the definition of the matrix $A_x$,
\[
T_xPV = PT_xV + A_xV
\]
with non negative $A_xV$ and $PT_xV$, we conclude therefore that the functions $A_xV$ and $PT_xV$ are also finite. 
\end{proof}

The following statement is the main technical point of our proof. 
 \begin{prop}\label{pr6-1} Under the hypotheses (A0), (B1) and (B2),  the function $V$  satisfies  \eqref{eq2-5a}.
\end{prop} 
The proof of this statement is given in Section~\ref{appendix}. Using this result we obtain

\begin{lemma}\label{lemma4-3} Under the hypotheses  (A0), (B1) and (B2), the function $V$ is super harmonic for the Markov chain $(X(t))$, $V (e)=1$  and 
 the function $(\Id-P)V$ is constant  in $    E$.  
 \end{lemma} 
\begin{proof} Indeed, recall that by Lemma~\ref{lemma4-1} and Lemma~\ref{lemma4-1}, the functions $PV$, $T_uPV$,  $A_uV$ and $PT_uV$ are finite on $E$. 
Since by \eqref{eq2-5a},  $T_uV \geq GA_uV$, we conclude therefore that the function $GA_u V$ is also finite and moreover potential for $(X(t))$ with 
\[
(\Id-P) GA_uV = A_u V.
\] 
The function  $PGA_uV \leq GA_uV$  is therefore also finite. Hence, from \eqref{eq2-5a} it follows that 
\begin{align*}
(\Id-P)T_u V &~=~ (\Id-P) V + (\Id-P) G A _u V ~=~ (\Id-P) V + A _u V \\ &~=~ (\Id-P) V  + T _u P V -  P T _u V,
\end{align*} 
and consequently, 
\[
 T _u(\Id-P)V ~=~ (\Id-P) V, \quad \forall u\in    E. 
\]
The last relation shows that the function $(\Id-P)V$ is constant on $E$. Remark now that $A_e = 0$, and 
\[
\E_e( H_1 = y) ~=~ p_H(e,y) ~=~ GA_e\1_{\{y\}} (u) ~=~ 0 \quad \forall y\in E. 
\]
Hence, $\P_e( H_1 = \vartheta) ~=~ p_H(e,\vartheta) ~=~ 1$ and consequently, $V (e) = \E_e({\cal T}_\vartheta) = 1$. Now, to prove that the function $V$ is super  harmonic for $(X(t))$ it is sufficient to notice that by Lemma~\ref{lemma4-1}, 
\[
PV (e) ~\leq~ 1 ~=~ V (e).
\]
Lemma~\ref{lemma4-3} is therefore proved. 
\end{proof} 

\begin{lemma}\label{lemma4-4} 
Suppose that the conditions (A0), (B1) and (B2) are satisfied and let $\E_\cdot(\tau_\vartheta) = + \infty$. Then the function $V$ is harmonic for $(X(t))$ and  
 \be\label{eq40-5}
 \lim_{n\to\infty} f_n(x) ~=~ V(x), \quad \forall x\in    E.
 \ee
\end{lemma}
\begin{proof} 
Indeed, recall that by Lemma~\ref{lemma4-1}, the function $V$ is super harmonic for $(X(t))$. Hence, by the Riesz decomposition theorem, 
 \[
 V ~=~ h + G \varphi 
 \]
 with  a harmonic for $(X(t))$ function $h ~=~\lim_n P^n V$ and $\varphi = (\Id-P)V$.  Recall moreover that by Lemma~\ref{lemma4-3}, the function $\varphi=(\Id-P)V$ is constant on $E$. Hence, by Fubini-Tonelli theorem,
 \[
 G\varphi(x) ~=~ \varphi (e) G\1(x) ~=~ \varphi (e)  \E_x(\tau_\vartheta),
 \]
 and consequently,  
 \[
V(x) ~=~   h +  \varphi (e)  \E_x(\tau_\vartheta) ~\geq~ \varphi (e)  \E_x(\tau_\vartheta).
 \] 
The last relations prove that  $\varphi = \varphi(e) \1 = 0$ and $V=h$  whenever $\E_\cdot(\tau_\vartheta) = + \infty$. 

To prove \eqref{eq40-5} recall that    by Lemma~\ref{lemma4-1}, if a subsequence $(f_{n_k})$ converges pointwisely on $E$, then the limit $f = \lim_{k\to\infty} f_{n_k}$ is a super harmonic function for $(X(t))$ and  satisfies the inequality  $V \leq f$. The function $V$ being harmonic, we conclude therefore that the function $f - V$ is super harmonic with $(f - V) (e) =  0$. The Markov chain $(X(t))$ being irreducible, by  the minimum principle, from this it follows that  $f - V  = 0$  and consequently, 
\[
\lim_{k\to\infty} f_{n_k} = V 
\]
for any convergent subsequence $(f_{n_k})$. The sequence $(f_n)$  being relatively compact with respect to the topology of point wise convergence, the last equality proves \eqref{eq40-5}. 
\end{proof}

\begin{lemma}\label{lemma4-5} If the conditions (A0), (B1) and (B2) are satisfied and $\E_\cdot(\tau_\vartheta) < \infty$, then the function $g(x) = \E_x(\tau_\vartheta)$ is  potential with $(\Id-P)g = \1$ and
\be\label{eq40-6}
g(u) ~\geq~ g (e) + \E_u(g( H_1), \; {\cal T}_\vartheta > 1), \quad \forall u\in    E. 
\ee
\end{lemma} 
\begin{proof} Indeed, if $\E_\cdot(\tau_\vartheta) < \infty$ then by Fubini-Tonelli theorem, $G\1 = \E_\cdot(\tau_\vartheta)$ and consequently,   the function $g(x) = \E_x(\tau_\vartheta)$ is  potential with $(\Id-P)g = \1$.  
Furthermore, for any $u\in E$, using the identity $T_u P = P T_u + A_u $ one gets  
\[
PT_u g + A_ug ~=~ T_uPg ~=~ T_u(g - \1) ~=~ T_u g - \1
\]
or equivalently, since the both terms of the left hand side of the above equality are positive and the right hand side is finite,
\[
(\Id-P)T_u g ~=~ \1 +  A_u g ~\geq~ 0. 
\]
For any $u\in E$, the function $T_ug$ is therefore super harmonic for $(X(t))$ and by Riesz decomposition theorem,
\[
T_ug(x) ~=~ \tilde{g}(x) + G(\1 + A_ug)(x), \quad \forall x\in E,  
\]
with $\tilde{g}= \lim_n P^n T_ug \geq 0$. Using the last relation with  $x=e$ we obtain 
\[
g(u) ~=~T_u g (e)  ~=~ \tilde{g}(e) +  G\1 (e) +  GA_u g (e)  ~\geq~ g (e) + GA_u g (e). 
\] 
According to the definition of the Markov process $( H_n)$, the last relation proves \eqref{eq40-6}. 

\end{proof} 

Now we are ready to complete the proof of Theorem~\ref{th2}.

 \subsection{Proof of Theorem~\ref{th2}.} Suppose that the conditions (A0), (B1) and (B2) are satisfied. Then by Lemma~\ref{lemma4-1} the function $V = \E_\cdot({\cal T}_\vartheta) = G_H\1$ is finite everywhere on $    E$ and by Lemma~\ref{lemma4-3}, $V (e)=1$.  This proves that the Markov chain $(H_n)$ is transient. 
 
Furthermore,  \eqref{eq2-5a} is proved by Proposition~\ref{pr6-1}.  If moreover, $\E_\cdot(\tau_\vartheta) = \infty$, then by by Lemma~\ref{lemma4-4}, the function $V$ is harmonic for the Markov chain $(X(t))$ and \eqref{eq2-5} holds.  The first tree assertions of Theorem~\ref{th2} are therefore proved. 
  
Suppose now that $\E_\cdot(\tau_\vartheta) < \infty$. Then by Lemma~\ref{lemma4-3}, the function $V$ is super harmonic for the Markov chain $(X(t))$ and the function $(\Id-P)V$ is constant on $    E$. Moreover, by Lemma~\ref{lemma4-5}, the function $g = G1 = \E_\cdot(\tau_\vartheta)$ satisfies the relations \eqref{eq40-6} and consequently, it is super harmonic for the ladder height process $( H_n)$ with $(\Id- P_H) g \geq g (e)$.  Hence, by the Riesz decomposition theorem, for any  $x\in    E$, 
\[
g(x) ~\geq~ g (e) G_H\1(x) +  \lim_n P_H^n g(x) ~\geq~ g (e) G_H\1(x) ~=~g(e) V(x)
\]
The function $g$ being potential for $(X(t))$, from the last inequality it follows that the super harmonic function $V$ is also potential for $(X(t))$. Since  the function $(\Id-P)V$ is constant on $E$,  this proves  that $V = c G\1$ with some $c > 0$, and consequently,  the functions $V$ and $g$ are proportional to each other. Finally, to get the equality $V = g/g (e)$ it is  sufficient to notice that $V (e) = 1$.

\section{Proof of Theorem~\ref{th3}}\label{proof-th3}
Suppose that $h:E\to\R_+$ is a harmonic function of $(X(t))$. Put $h(\vartheta) = 0$. Then using the definition of the ladder height process $(H_n)$ and \eqref{e1-2}  with $u=x$ and $y= e$  one gets
\[
h(x)  ~=~ h(e) +  \E_x\bigl(h(H_1)\bigr), \quad x\in E.
\]
This relation shows that any harmonic function $h$ of $(X(t))$ is  super harmonic for $(H_n)$ with $(\Id - P_H) h = h(e)$. By Riesz decomposition theorem, from this it follows that 
\[
h ~=~h(e) G_H\1 + \tilde{h} 
\]
where the function $\tilde{h}= \lim_n P_H^n h$ is harmonic for $(H_n)$. The last relation proves \eqref{eq2-7n} because  by Fubini-Tonelli theorem, 
$
G_H\1(x) = \E_x({\cal T}_\vartheta) = V(x)$
The first assertion of Theorem~\ref{th3} is therefore proved. 

Suppose now that $\E_\cdot(\tau_\vartheta) = +\infty$. Then  by Theorem~\ref{th2}, the function $V$ is harmonic for $(X(t))$. By \eqref{eq2-7n}, for any harmonic function $h$, the function $\tilde{h} = h - h(e)V$ is therefore also harmonic for $(X(t))$ with $\tilde{h}(e)=0$ because  $V(e)=1$.  By the minimum principle,  this proves that $\tilde{h} = 0$ and consequently, the function $h$ is proportional to $V$. 

If $\E_\cdot(\tau_\vartheta) < +\infty$, then by Theorem~\ref{th2}, the function $V$ is potential for $(X(t))$ and hence, for any harmonic function $h$, the function $\tilde{h} = h - h(e)V$ is non zero.

\section{Proof of Proposition~\ref{pr6-1}}\label{appendix} Before proving this proposition in our general setting let us consider the main ideas of the proof in the case when (B2) holds with the equality. From now on throughout this section the assumption (A0) and (B1) are assumed satisfied. 

If (B2) holds with the equality, the ladder height process $(H_n)$ is given by \eqref{e1-10}, the equation \eqref{eq2-5a} is equivalent to \eqref{e1-200} for $h=V$, and the transition probabilities of $(H_n)$ satisfy the following relations  
\[
p_H(x,y) = GA_{x}\1_{\{ y \}} (e) ~=~ \begin{cases} \sum_{z\in E} G(e,z) p(z\star x, y) &\text{if $y\in E\setminus(E\star x)$,}\\
0 &\text{otherwise.} 
\end{cases} 
\]
Hence, for any $u, x, y \in E$,
\[
p_H(x\star u, y\star u) = p_H(x,y), 
\]
and consequently, letting for $u\in E$,
\[
\nu_u = \inf\{n\geq 0 : H_n \not\in E\star u\}
\]
and using the inequality ${\cal T}_\vartheta \geq \nu_u$, one gets  
\begin{align*}
\P_{x\star u} ({\cal T}_\vartheta > n) = &\P_{x\star u}(\nu_u > n) \\&\hspace{0.5cm}+ \sum_{k=1}^n \sum_{z\in E\setminus(E\star u)} \P_{x\star u}(H_k = z, \, \nu_u = k) \P_z({\cal T}_\vartheta > n-k) 
\end{align*} 
with
\[
\P_{x\star u}(\nu_u > n) = \P_x({\cal T}_\varphi > n). 
\]
From the two last relations using the identities
\[
V(x) = \E_{x}({\cal T}_\vartheta) = \sum_{n=0}^\infty \P_{x}({\cal T}_\vartheta > n) 
\]
it follows that 
\[
V(x\star u) = V(x) + \E_{x\star u}\bigl(V(H_{\nu_u}), \, \nu_u < {\cal T}_\vartheta \leq +\infty\bigr)
\]
and using moreover \eqref{e1-10} we conclude that  
\[
V(x\star u) = V(x) + \E_{x\star u}\bigl(V(X(t_{\nu_u})), \, \nu_u < {\cal T}_\vartheta \leq +\infty\bigr). 
\]
To complete the proof of \eqref{e1-200} for $V=h$ it is therefore sufficient to show that 
\[
\E_{x\star u}\bigl(V(X(t_{\nu_u})), \, \nu_u < {\cal T}_\vartheta \leq  +\infty\bigr) = \E_{x\star u}\bigl(V(X(\eta_u)), \; \eta_u < \tau_\vartheta \leq + \infty\bigr). 
\]
The last relation  follows from \eqref{e1-10} because the event  $\{\nu_u < {\cal T}_\vartheta \leq +\infty\}$  is equivalent to the event  $\{\eta_u < \tau_\vartheta \leq + \infty\}$, and on the event $\{\eta_u < \tau_\vartheta \leq + \infty\}$, the stopping times $t_1, \ldots, t_{\nu_u}$ are finite with $t_{\nu_u}= \eta_u$. 
In the case when (B2) holds with the equality, Proposition~\ref{pr6-1} is therefore proved. 

\medskip 

To prove \eqref{eq2-5a} in our general setting we construct two Markov chains $(  H_n)$ and $(  H_n^u)$ on the same probability space,  with the same transition probabilities  \eqref{eq2-3} and \eqref{eq2-4p}, and such that $  H_n \star u =   H_n^u$ for any $n\leq {\cal T}_\vartheta = \inf\{k > 0: H_k = \vartheta\}$. For this we need the following lemmas. 
\begin{lemma}\label{lem60-0} Under the hypotheses (A0), (B1) and (B2), for any $x,u\in E$, 
\[
A_{x\star u} = A_x T_u + T_xA_u
\]
\end{lemma} 
\begin{proof} Indeed, according to the definition of the matrices $A_x, x\in E$, 
\begin{align*}
A_x T_u + T_xA_u &=~ T_xPT_u - PT_xT_u + T_xT_uP - T_xPT_u  ~=~ T_xT_uP - PT_xT_u  \\&=~ T_{x\star u} P -  PT_{x\star u} ~=~ A_{x\star u} 
\end{align*} 
\end{proof} 
\noindent
Using this lemma we get 
\begin{lemma}\label{lem60-1} Under the hypotheses (A0), (B1) and (B2), transition probabilities of the ladder process $(H_n)$  
satisfy the following relations 
\begin{equation}\label{eq60-1} 
p_H(x\star u, y\star u) ~=~p_H(x,y) + GT_xA_u\1_{\{ y\star u\}}  (e) ~\geq~ p_H(x,y), \quad \forall x,y, u \in E, 
\end{equation} 
\end{lemma} 
\begin{proof} By Lemma~\ref{lem60-0}, from the definition of the transition probabilities $p_H(x,y)$, $x,y\in E$, it follows that for any $x,y\in E$, 
\[
p_H(x \star u, y\star u) ~=~ GA_{x \star u   }\1_{\{ y\star u\}}  (e) ~=~ GA_x T_u\1_{\{ y\star u\}}  (e) + GT_xA_u\1_{\{ y\star u\}}  (e)
\]
Since $T_u\1_{\{ y\star u\}} = \1_{\{ y\}}$,  $GA_x\1_{\{ y\}}  (e) = p_H(x,y)$ and $GT_xA_u\1_{\{ y\star u\}}  (e)\geq 0$, the last relation proves \eqref{eq60-1}. 
\end{proof} 

Now we are ready to introduce the process $(H^u_n,   H_n)$. We define $(H^u_n,   H_n)$ as a Markov chain on 
\[
\tilde{E} ~=~ \bigl \{ (y\star u, y), \quad y \in E\bigr\} \cup \bigl\{(y,\vartheta), y\in E\bigr\} \cup \bigl\{(\vartheta,\vartheta)\bigr\}
\]
with absorbing state $(\vartheta,\vartheta)$ and  transition probabilities 
\[\P_{(y^u,y)} \bigl((H^u_1,H_1) = (z^u,z)\bigr) = \tilde{p}_H\bigl( (y^u,y), (z^u, z) \bigr), \quad (y^u,u), (z^u,z)\in \tilde{E}
\]
such that for $y, z\in E$, 
\[
\tilde{p}_H\bigl( (y\star u,y), ( z\star u, z) \bigr) ~=~ p_H(y, z),
\]
\[
\tilde{p}_H\bigl( (y\star u,y), ( z\star u, \vartheta) \bigr) ~=~ p_H(y\star u, z\star u) - p_H(y,  z) ~=~ GT_yA_u\1_{\{ z\star u\}}  (e)
\]
\[
\tilde{p}_H\bigl( (y\star u,y), ( z,\vartheta) \bigr) ~=~
p_H(y\star u, z)  \quad \text{if $ z\in E\setminus E\star u$,}
\]
\[
\tilde{p}_H\bigl( (y\star u,y), (\vartheta,\vartheta) \bigr) ~=~ 1 - \sum_{ z\in E} p_H(y\star u,  z) 
\]
and 
\[
\tilde{p}_H\bigl( (y,\vartheta), ( z, z') \bigr) =\begin{cases} p_H(y, z) &\text{if $ z'=\vartheta$ and  $ z\in E$,}\\
p_H(y,\vartheta)  &\text{if $ z= z'=\vartheta$,} \\
0 &\text{otherwise} 
\end{cases} 
\]
By Lemma~\ref{lem60-1}, under the hypotheses (A0), (B1) and (B2), the coefficients of the matrix $(\tilde{p}_H(v,v'); \; v,v'\in \tilde{E})$ are non negative and  for any $v\in \tilde{E}$, 
\[
\sum_{v'\in\tilde{E}} \tilde{p}_H(v,v') = 1.
\]
The Markov chain $(H_n^u,H_n)$ is therefore well defined.  
For $(y^u,y)\in \tilde{E}$, we denote respectively by $\P_{y^u\!,y}$ and $\E_{y^u\!,y}$ the probability measure and the corresponding expectation given that $H^u_0 = y^u$ and $  H_0=y$.

Remark that according to the above definition, $(H_n^u)$ and $(H_n)$ are two Markov chains with the same transition probabilities defined by \eqref{eq2-3} and \eqref{eq2-4p}. For any $(y^u,y)\in \tilde{E}$, and any integrable and $\sigma( (H_n))$-measurable random variable $W$, we have therefore 
\[
\E_{y^u\!,y}(W) = \E_y(W),
\]
and similarly, for any integrable and $\sigma((H^u_n))$- measurable random variable $W^u$, 
\[
E_{y^u\!, y}(W^u) = \E_{y^u}(W^u).  
\]
In particular, for ${\cal T}_\vartheta  = \inf\{k > 0: H_k = \vartheta\}$ and ${\cal T}^u_\vartheta  = \inf\{n > 0 : H^u_n = \vartheta\}$, 
\be\label{eq60-1p} 
V(x\star u) = \E_{x\star u,x} ({\cal T}^u_\vartheta ), \quad  V(x) = \E_{x\star u, x}({\cal T}_\vartheta) \quad \text{and} \quad V(z) = \E_{z,\vartheta}({\cal T}_\vartheta^u) 
\ee
Remark moreover, that for any $x\in E$,  $\P_{x\star u,x}$-almost surely, 
\[
H^u_n = H_n\star u \in E\star u \subset E, \quad \forall n < {\cal T}_\vartheta. 
\]
The stopping times ${\cal T}_\vartheta$ and 
${\cal T}^u_\vartheta$ satisfy therefore the inequality 
\[
{\cal T}^u_\vartheta \geq {\cal T}_\vartheta,
\]
and using the same arguments as in the previous particular case, we obtain
\begin{align*}
\P_{x\star u,x}({\cal T}^u_\vartheta \!> \!n) &= \P_{x\star u, x} ({\cal T}_\vartheta \!> \!n) \\ &\hspace{1cm}+\sum_{k=1}^n \sum_{z\in E} \P_{x\star u, x}( H_k^u \!= z, {\cal T}_\vartheta = k) \P_{z,\vartheta}({\cal T}^u_\vartheta > n-k). 
\end{align*}
From the last relation using \eqref{eq60-1p} one gets 
\be\label{eq60-2}
V(x\star u) = V(x) + \E_{x\star u,x}\left(V(H^u_{{\cal T}_\vartheta}), \; {\cal T}_\vartheta < {\cal T}_\vartheta^u \leq +\infty\right).
\ee
Now, to complete the proof of \eqref{eq2-5a} it is sufficient to show that 
\be\label{eq60-3} 
\E_{x\star u,x}\left(V(H^u_{{\cal T}_\vartheta}), \; {\cal T}_\vartheta < {\cal T}_\vartheta^u \leq +\infty\right) = GA_uV(x), \quad \forall x, u\in E. 
\ee
To prove this equality we use the following two lemmas. 

\medskip 

\begin{lemma}\label{lem60-2} Under the hypotheses (A0), (B1) and (B2), for any $u\in E $, the function 
\[
F_u(x) = E_{x\star u,x}\left(V(H^u_{{\cal T}_\vartheta}), \; {\cal T}_\vartheta < {\cal T}_\vartheta^u \leq +\infty\right) 
\]
 is potential for the ladder process $(H_n)$ with 
\[
(\Id-P_H) F_u(x) ~=~ GT_xA_uV(e) , \quad x\in E. 
\]
\end{lemma} 
\begin{proof} Denote 
\[
\varphi_u(x) = \E_{x \star u   ,x}(V(  H ^u _1), \;   {\cal T}_\vartheta = 1 < {\cal T}^u_\vartheta ), \quad x\in E. 
\]
Then by Markov property, for any $k\geq 1$, $x\in E$, 
\begin{align*}
&\E_{x \star u   ,x}\left(V(  H ^u _k), \; {\cal T}_\vartheta = k < {\cal T}^u_\vartheta \right) \nonumber \\&= \sum_{y\in E } \P_{x \star u   ,x}\Bigl(  (H^u_{k-1}, H_{k-1}) = (y\star u, y)\Bigr)\E_{y \star u   ,y}\!\left(V(  H ^u _1), \;  {\cal T}_\vartheta = 1 < {\cal T}^u_\vartheta \right)  \\
&= \sum_{y\in E } \P_{x}\!\left(  H_{k-1} = y\right) \varphi_u(y)  ~=~ P_H^{k-1} \varphi_u (x)  
\end{align*} 
and consequently, for any $x\in E$, 
\[
F_u(x) = \E_{x \star u   ,x}\left(V(  H ^u _{{\cal T}_\vartheta} ), \,{\cal T}_\vartheta < {\cal T}^u_\vartheta \leq +\infty \right)  = \sum_{k=1}^\infty P_H^{k-1} \varphi_u(x) = G_H \varphi_u(x).  
\]
The function $F_u$ is therefore potential for $(H_n)$ with $(\Id - P_H)F_u = \varphi_u$. To complete the proof of our lemma it is now sufficient to show that 
\be\label{eq60-4} 
\varphi_u(x) ~=~ GT_xA_uV(e), \quad \forall x\in E. 
\ee
To prove these relations we notice that, by definition of the transition probabilities $\tilde{p}((x\star u,x), (y,\vartheta))$ and $p_H(x\star u, y)$, for any $x\in E$, 
\begin{align*}
\varphi_u(x)  &=~ \E_{x \star u   ,x}(V(  H ^u _1), \;   {\cal T}_\vartheta = 1 < {\cal T}^u_\vartheta ) ~=~ \sum_{y\in E} \tilde{p}((x\star u,x), (y,\vartheta)) V(y) \\
&= \sum_{y\in E} (p_H(x\star u, y\star u) - p_H(x,  y)) V(y\star u)  + \!\!\!\sum_{y\in E\setminus(E\star u)} \!\!p_H(x\star u, y) V(y) \\
&= \sum_{y\in E} p_H(x\star u, y) V(y)  -  \sum_{y\in E}  p_H(x,  y)V(y\star u) \\
&= P_HV(x\star u) - P_HT_uV(x) ~=~  GA_{x\star u}V(e) - GA_xT_uV(e) 
\end{align*}
The last relation combined with Lemma~\ref{lem60-0} proves \eqref{eq60-4}. 
\end{proof}

\begin{lemma}\label{lem60-3} Under the hypotheses (A0), (B1) and (B2), for any $u\in E $, the function $GA _u V$ is potential for the ladder height process $(H_n)$ and for any $x\in E$, 
\be\label{eq60-5} 
(\Id-P_H)GA _u V(x) = GT_xA_uV(e)
\ee 
\end{lemma} 
\begin{proof} 
Recall  that  according to the definition of the operator $A_x$, 
\[
T_xP = PT_x + A_x. 
\]
From this identity it follows that for any $n\geq 1$, 
\[
T_xP^n = PT_x P^{n-1} + A_xP^{n-1}  =  P^nT_x +  \sum_{k=1}^{n}P^kA_xP^{n- k}.
\]
and consequently, 
\be\label{eq60-6}
P^nA_u V (x) ~=~ T_xP^nA_u V(e)  ~=~ P^nT_xA_uV(e)  + \sum_{k=1}^{n} P^{k-1}A_xP^{n- k}A_u V (e)
\ee
Consider a sequence of operators $B_n$ defined on the set of non negative functions $\{\varphi : E \to \R_+\}$ by 
\[
B_n\varphi (x)  ~=~ P^{n-1}A_x\varphi(e) ~=~\sum_{y\in E} b_n(x,y)\varphi(y), \quad x\in E, 
\]
with 
\[
b_n(x,y) = P^{n-1}A_x\1_y(e) \geq 0, \quad \forall x,y\in E, 
\]
and let 
\[
  \varphi_{u,n}(x) = P^nT_xA_uV(e), \quad x\in E. 
\]
With these notations,  letting 
\[
\varphi_u(x) = GT_xA_uV(e)
\]
we obtain 
\be\label{eq60-7} 
\varphi_u(x) ~=~ \sum_{n=0}^\infty \varphi_{u,n}(x), \quad \forall x\in E.
\ee
Moreover, by definition of the transition probabilities of the process$(H_n)$, 
\be\label{eq60-8}
P_H\varphi(x) ~=~ \E_x(\varphi(H_1)) ~=~ GA_x\varphi(e) ~=~ \sum_{n=1}^\infty B_n\varphi(x) , 
\ee
and using \eqref{eq60-6}, for any $n\geq 1$, we get 
\begin{align*}
P^nA_u V (x) &=~ \varphi_{u,n}(x) + \sum_{k=1}^{n} B_kP^{n- k}A_u V (x) \\
&= \varphi_{u,n}(x)  + \sum_{k=1}^{n} \sum_{y\in E} b_k(x,y) P^{n- k}A_u V (y).
\end{align*}
Since for $k=n$, $P^{n-k}A_u V(x) = A_uV(x) = \varphi_{u,0}(x)$, the iterates of the last relations provide the following relations   
\begin{align*}
P^nA_u V(x)  &= \varphi_{u,n}(x)  + \!\sum_{m=1}^n \,\sum_{y_1,\ldots,y_m\in E}  \sum_{\substack{k_1,\ldots k_m\geq 1:\\k_1+\cdots + k_m \leq n} }\!\!b_{k_1}(x,y_1) \cdots b_{k_m}(y_{m-1},y_m) \\ 
&\hspace{7.1cm}\times \varphi_{u,n-(k_1+\cdots +k_m)} (y_m) \\
&=\varphi_{u,n}(x)  + \!\sum_{m=1}^n  \sum_{\substack{k_1,\ldots k_m\geq 1:\\k_1+\cdots + k_m \leq n} }\hspace{-0.5cm}B_{k_1}\ldots B_{k_m}\varphi_{u,n-(k_1+\cdots +k_m)}(x)
\end{align*} 
When combined with \eqref{eq60-7} and \eqref{eq60-8} the above relations prove that for any $x\in E$, 
\begin{align*}
GA_uV(x) = \sum_{n=0}^\infty P^nA_uV(x) ~=~ \varphi_u(x) + \sum_{m=1}^n P_H^m \varphi_u(x) ~=~ G_H\varphi_u(x) 
\end{align*} 
and consequently,  the function $GA_uV$ is potential for the ladder height process $(H_n)$ with $(\Id-P_H)GA_uV(x) = \varphi_u(x) = GT_xA_uV(e)$. 
\end{proof} 
\medskip

Now we are ready to complete the proof of \eqref{eq2-5a}. The functions $F_u = T_uV - V$ and $GA_uV$ being potential for the ladder process $(H_n)$ with the same function $\varphi_u = (\Id-P_H)(T_u V - V) = (\Id-P_H) GA_uV$ we conclude that $T_uV - V = G_H \varphi_u = GA_uV$. The last equality is equivalent to \eqref{eq2-5a}. Proposition~\ref{pr6-1} is therefore proved.

\section{Proof of Proposition~\ref{prC-1}.}\label{proof_of_prC-1} We begin the proof of Proposition~\ref{prC-1} with the following preliminary results.
\subsection{Preliminary results} In this subsection, the conditions (C0') and (C1') are assumed satisfied, but instead of the assumption (C2) we will assume that 

{\em 
\begin{enumerate}
\item[(C2')] the  step generating function 
\[
R(\alpha) ~=~ \sum_{x\in\Z^d} e^{\alpha\cdot x} \mu(x) 
\]
is finite in a neighborhood of the set $D = \{\alpha\in\R^d :~R(\alpha)\leq 1\}$ on $\R^d$ and 
\[ 
\sum_{x\in \Z^d} x \mu(x) ~\not=~0. 
\]
\end{enumerate} 
}
Under the hypotheses (C0'), (C1') and (C2), the results of this subsection will be applied next to truncated  versions of the original random walk $(X(t))$. By construction, the truncated  versions of $(X(t))$  will satisfy (C2') instead of (C2).  As before, we denote by $G(x,y), \, x,y\in E$, the Green function of $(X(t))$
 \[
 G(x,y)  ~=~ \sum_{n=0}^\infty \P_x(X(n) = y).
 \] 
The main result of this subsection is the following statement. 
\begin{prop}\label{pr7-100} Under the hypotheses (C0') and (C1') and (C2'),  for any $u,v\in\overline{\cal C}\times\overline{\cal C}$ and any sequences $(u_n),(v_n)\in E^\N$, with $\lim_n u_n/n = u$ and $\lim_n v_n/n = v$, 
\be\label{eq70-1}
\lim_{n\to\infty}\frac{1}{n} \log G(u_n, v_n) ~=~ - \sup_{\alpha\in D} \langle \alpha, v - u\rangle,
\ee
\end{prop} 
Before proving this result let us consider its following  straightforward consequence.   
\begin{cor}\label{corCC-1}  Under the hypotheses (C0') and (C1') and (C2'), for any $u\in E$ and uniformly on $x\in E$, 
\be\label{eq70-2}
\liminf_{n\to\infty} \frac{1}{n}\log Q(x, x+nu) ~\geq~ - \sup_{\alpha\in D} \langle \alpha, u\rangle 
\ee
and
\be\label{eq70-3} 
\liminf_{n\to\infty} \frac{1}{n}\log Q(x+nu, x) ~\geq~ - \sup_{\alpha\in D} \langle \alpha, - u\rangle 
\ee
\end{cor} 
\begin{proof} Indeed, consider a homogeneous random walk $(S(t))$ on $\Z^d$ with transition probabilities $
\P_x(S(1) = y) = \mu(y-x), \quad x,y\in \Z^d$. 
Our random walk $(X(t))$ is a copy of the random walk $(S(t))$ killed upon the time $\tau_\vartheta = \inf\{ t\geq 1:~ S(t) \not\in E\}$ :
\[
X(t) ~=~\begin{cases} S(t), &\text{if $t <\tau_\vartheta$},\\ 
\vartheta &\text{otherwise.} 
\end{cases} 
\]
Because of the assumption (C2'), the random walk $(S(t))$ is transient on $\Z^d$ and Its Green function 
\[
G_S(x,y) ~=~ \sum_{n=0}^\infty\P_x(S(n)=y), \; x,y\in\Z^d
\] 
satisfies the inequalities 
\[
G(x,y) ~\leq~G_S(x,y) ~=~ G_S(0, y-x), \quad \forall x,y\in E.
\] 
Hence, for any $x, u\in E$, 
\[
Q(x, x+nu) ~\geq~ Q(0, nu) ~=~ G(0, nu)/G(nu,nu) ~\geq~ G(0,nu)/G_S(0,0).
\]
When combined with \eqref{eq70-1}, the last inequality proves \eqref{eq70-2}. The proof of \eqref{eq70-3} is quite similar. 
\end{proof} 

To prove Proposition~\ref{pr7-100} it is sufficient to show that for any $u,v\in\overline{\cal C}\times\overline{\cal C}$, and  $(u_n),(v_n)\in E^\N$ with $\lim_n u_n/n = u$ and $\lim_n v_n/n = v$, the upper bound 
\be\label{eq70-4}
\limsup_{n\to\infty} \frac{1}{n} \log G(u_n, v_n) ~\leq~ - \sup_{\alpha\in D} \langle \alpha, v - u\rangle, 
\ee
and the lower bound 
\be\label{eq70-5}
\liminf_{n\to\infty} \frac{1}{n} \log G(u_n, v_n) ~\geq~ - \sup_{\alpha\in D} \langle \alpha, v - u\rangle.
\ee
hold. The following lemma proves that these limits are well defined for all $u,v\in\overline{\cal C}$. 
\begin{lemma}\label{lem70-1} Under the hypotheses (C1'), for any $u\in\overline{\cal C}$, there is a sequence $(u_n)\in E^\N$  such that $u_n/n \to u$ as $n\to\infty$. 
\end{lemma}
\begin{proof}
For $u\in\overline{\cal C}$ we denote by $[u]\in E$ a nearest to $u$ point of $E$, i.e.  such that  
\[
\min_{v\in E} |u-v| = | u - [u]|, 
\]
To prove our lemma it is sufficient to show that for any $u\in\overline{\cal C}$, 
\be\label{eq70-6} 
\lim_{n\to\infty} \left| u - \frac{ [nu]}{n}\right| = 0. 
\ee
Let $u\in\overline{\cal C}$. Then for any $\eps > 0$ there is $u_\eps \in {\cal C}$ such that $|u-u_\eps| < \eps$. Moreover, the set ${\cal C}$ being open, there is $0 < r_\eps < \eps$ for which an open ball $B(u_\eps,r_\eps)$ centered at $u_\eps$ and having radius $r_\eps$ is included to $\overline{\cal C}$. For any $n\geq 1$, we get therefore 
\[
B(nu_\eps, n r_\eps) \subset \overline{\cal C}.
\]
Since $E = \overline{\cal C}\cap \Z^d$,  any open ball $B(a, r)\subset\overline{\cal C}$ with $r > d$ contains some point of the set $E$, and consequently, for $n > d/r_\eps$, 
\[
 \left| u - \frac{ [nu]}{n}\right|   \leq |u - u_\eps| + r_\eps ~\leq |u - u_\eps|  + \eps
\] 
Letting therefore first $n\to\infty$ and next $\eps\to 0$ one gets  \eqref{eq70-6}. \end{proof} 

The upper bound~\eqref{eq70-4} follows from the following statement:
\begin{lemma}\label{lem70-2} Under the hypotheses (C0'), (C1'), (C2'), for any $\alpha\in D$ and  $(u_n),(v_n)\in E^\N$ with $\lim_n u_n/n = u$ and $\lim_n v_n/n = v$, 
\be\label{eq70-7}
\limsup_{n\to\infty} \frac{1}{n} \log G(u_n, v_n) ~\leq~ - \langle\alpha, v-u\rangle.
\ee
\end{lemma}
\begin{proof} The proof of this lemma uses the method of the exponential change of measure. For $\alpha\in D$, the twisted measure 
$
\mu_\alpha(x) = \exp(\langle \alpha, x\rangle)\mu(x)$ 
is substochastic on $\Z^d$. Consider a twisted random walk $(X_\alpha(t))$ on $E$ with transition probabilities  
$
\P_x(X_\alpha(1) = y) = \mu_\alpha(y-x), \;x,y\in E$.
Since clearly, 
\[
\P_x(X_\alpha(t) = y) ~=~ \exp(\langle\alpha, y-x\rangle) \P_x(X(t)=y)
\]
for all $t\in\N$ and $x,y\in E$, the Green function $G_\alpha(x,y)$ of the twisted random walk $(X_\alpha(t))$ satisfies the identities 
\[
G_\alpha(x,y) ~=~ \exp(\langle\alpha, y-x\rangle) G(x,y), \quad x,y\in E, 
\]
and in particular, for any $y\in E$, 
\[
G_\alpha(y,y) = G(y,y) ~\leq~ G_S(0,0) 
\] 
where $G_S(x,y)$ denotes the Green function of the homogeneous random walk $(S(t))$ on $\Z^d$ with transition probabilities 
$
\P_x(S(1) = y) ~=~\mu(y-x), \; x,y\in \Z^d$. 
The homogeneous random walk being transient, from this it follows that  
\begin{align*}
G(x,y) &\leq \exp(- \langle\alpha, y-x\rangle) G_\alpha(x,y) \\ &\leq~ \exp(- \langle\alpha, y-x\rangle) \P_x(X_\alpha(t) = y \; \text{for some} \; t \geq 0) G_S(0,0) \\ 
&\leq~ \exp(- \langle\alpha, y-x\rangle) G_S(0,0),
\end{align*} 
and the last inequality implies \eqref{eq70-7}. 
\end{proof} 

To prove the lower bound \eqref{eq70-5} we need the following lemmas. The first lemma is a straightforward consequence of the communication condition (C0'). Recall that 
under the hypotheses (C0'), there is a finite set 
\[{\cal E}_0\subset\supp(\mu) = \{x\in\Z^d: \mu(x) > 0\}, 
\]
such that for any  $x,y\in E$ there exists a sequence $x_0,x_1,\ldots, x_m\in E$  with 
$x_0 = x$, $x_m =y$, $m\leq \kappa |y- x|$, and $x_j-x_{j-1}\in{\cal E}_0$ for all $j=1,\ldots,k$. Denote 
\[
\delta = \min_{x\in{\cal E}_0} \mu(x) ~>~ 0. 
\]
\begin{lemma}\label{lem70-3} Under the hypotheses (C0'), for any $x,y\in E$, there is $m\leq \kappa |y-x|$ such that 
\[
\P_x(X(m)=y) ~\geq~\delta^{\kappa|y-x|} 
\]
\end{lemma} 
\begin{proof} Because of the assumption (C0'), for any  $x,y\in E$ there is a sequence $x_0$, $x_1$, \ldots, $x_m{\in}E$  with 
$x_0 = x$, $x_m =y$ and $m\leq \kappa |y- x|$ such that $x_j-x_{j-1}\in{\cal E}_0$ for all $j=1,\ldots,m$. Hence, for any $ j=1,\ldots, m$, 
one gets 
\[
\P_{x_{j-1}}(X(1) = x_j) ~=~ \tilde\mu(x_j - x_{j-1}) \geq \delta  > 0,
\]  
and consequently, 
\[
\P_{x}(X(m) = y) ~\geq~ \delta^m ~\geq~\delta^{\kappa |y-x|} 
\]
\end{proof} 

For a real number $s>0$ we denote by $[s]$ the integer part of $s$. The lower bound \eqref{eq70-5} will be proved in two steps. First we establish this bound for the points $u\not=v$ in the interior ${\cal C}$ of the cone $\overline{\cal C}$, and next we extend this bound to the whole cone $\overline{\cal C}$ by showing that  the left and the right sides of \eqref{eq70-5} are continuous on $\overline{\cal C}\times\overline{\cal C}$. To prove the lower bound for $u,v\in{\cal C}$, we need 
\begin{lemma}\label{lem70-4} For any $u\not= v$, $u,v\in{\cal C}$, there is a positive real number $T_{u,v} >0$ such that for any sequence $(u_n)\in E^\N$ with $\lim_n u_n/n = u$,  
\be\label{eq70-8} 
\lim_{\eps \to 0} \liminf_{n\to\infty} \frac{1}{n} \log\P_{u_n}( |X([T_{u,v}n]) - v n| \leq \eps n) ~\geq~ - \sup_{\alpha\in D} \langle \alpha, v - u\rangle 
\ee
\end{lemma} 
\begin{proof}
Under the hypotheses (C0'), the function $\log R$ is strictly convex, the set $D = \{\alpha :~R(\alpha)\leq 1\}$ is convex and compact, the gradient $\nabla\log R(\alpha)$ exists everywhere on  $\R^d$ and does not vanish on the boundary $\partial D = \{\alpha\in\R^d :~ R(\alpha) = 1\}$, and the mapping $\alpha \to \nabla\log R(\alpha) = \nabla R(\alpha)/R(\alpha)$ determines a homeomorphism between $\partial\tilde{D}$ and the unit sphere $S^d$ in $\R^d$ (see \cite{Hennequin}) For any non zero vector $u\in\R^d$ there is therefore a unique point $\alpha_{u} \in \partial D$ for which the gradient vector $\nabla R(\alpha_{u} )$ is proportional to $u$:
\be\label{eq70-9} 
\frac{\nabla R(\alpha_u)}{|\nabla R (\alpha_u)|} = \frac{u}{|u|}
\ee
and 
\be\label{eq70-10}
\sup_{\alpha\in D} \langle \alpha, u\rangle = \langle \alpha_u, u\rangle. 
\ee
To prove our lemma we use large deviation estimates (see Mogulkii's theorem in the book of Dembo and Zeitouni~\cite{D-Z}) for scaled processes 
\[
Z_n(t) = \frac{1}{n} S([nt]), \; t\in \R_+
\]
where $(S(t))$ is a homogeneous random walk on $\Z^d$ with transition probabilities $\P_x(S(1)=y) = \mu(y-x)$ and $[nt]$ denotes the integer part of $nt$. Recall that a function $\phi:[0,T]\to \R^d$ is absolutely continuous if there existes an integrable function $\dot\phi :[0,T]\to \R^d$ such that 
\[
\phi(t) = \phi(0) + \int_0^t \dot\phi(s) \, ds, \quad \forall t\in[0,T]. 
\] 
The lower large deviation bound of  Mogulkii's theorem proves that for any $T>0$, $\eps > 0$ and any absolutely continuous function $\phi: [0,T]\to\R^d$, if a sequence of points $(u_n)\in (\Z^d)^\N$ converges to $\phi(0)$ then 
\[
 \liminf_{n\to\infty} \frac{1}{n}\log \P_{u_n}\left( \sup_{t\in[0,T]}|Z_n(t) - \phi(t)| < \eps \right)  ~\geq~ - \int_0^T( \log R)^*(\dot\phi(t)) \, dt 
\]
where $\P_{u_n}$ denotes the conditional probability given that $S(0) = u_n$ and $(\log R)^*$ denotes the convex conjugate of the function $\log R$ defined by  
\[
(\log R)^*(w) ~=~\sup_{\alpha\in \R^d} \left(  \langle \alpha, w\rangle - R(\alpha) \right), \; w\in\R^d.
\]
When applied with an affine function $\phi(t) = u + wt$, $t\in[0,T]$, this inequality implies that 
\[
\lim_{\substack{\eps \to 0 \\ \eps > 0}}\liminf_{n\to\infty} \frac{1}{n}\log \P_{u_n}( \sup_{t\in[0,T]}| Z_n(t)  - u - wt | < \eps) ~\geq~ - T (\log R)^*(w). 
\]
Recall moreover that if for a given $w\in\R^r$ there is point $\alpha\in\R^d$ such that $\nabla (\log R)(\alpha) = w$ and $R(\alpha) = 1$, then $(\log R)^*(w) = \langle \alpha, w\rangle - \log R(\alpha) = \langle \alpha, w\rangle$ and consequently, from the last inequality on gets 
\[
\lim_{\substack{\eps \to 0 \\ \eps > 0}}\liminf_{n\to\infty} \frac{1}{n}\log \P_{u_n}( \sup_{t\in[0,T]}| Z_n(t)  - u - wt | < \eps) ~\geq~ - T \langle \alpha, w\rangle. 
\]
Consider now two points $u\not= v$, $u,v\in{\cal C}$ and let a sequence $(u_n)\in E^\N$ be such that $\lim_n u_n/n = u$. Then  for the affine function $\varphi_{u,v}(t) = u + w_{v-u}t$ with $w_{v-u} = \nabla R(\alpha_{v-u})$ and $T = T_{v-u} = |v-u|/|\nabla R(\alpha_{v-u})|$, from the last inequality it follows that  
\begin{align}
\lim_{\substack{\eps \to 0 \\ \eps > 0}}\liminf_{n\to\infty} \frac{1}{n}\log \P_{u_n}&\left( \sup_{t\in[0,T_{v-u}]} \left| Z_n(t) - \varphi_{u,v}(t) \right| < \eps\right) \nonumber\\ 
&\hspace{2cm}\geq - \langle\alpha_{v-u}, v-u\rangle \label{eq70-11} 
\end{align}
because according to the definition of the mapping $u \to \alpha_u$, $T_{v-u}w_{v-u} = v-u$ and $R(\alpha_{v-u}) = 1$. The set ${\cal C}$ being open and convex, for small $\eps > 0$ and large $n$, on the event 
\[
\left\{ \sup_{t\in[0,T_{v-u}]} \left| Z_n(t) - u - w_{v-u}t \right| < \eps\right\},
\]
our random walk $(S(t))$ starting at $u_n$ does not exit from $E= \overline{\cal C}\cap \Z^d$. This proves that the left hand side of \eqref{eq70-11} is equal to the similar probability for the  random walk $(X(t))$ starting at $X(0) = u_n$ and killed upon the time $\tau_\vartheta$ of the first exit from $E$, and consequently,  
\begin{multline}\label{eq70-12}
\lim_{\substack{\eps \to 0 \\ \eps > 0}}\liminf_{n\to\infty} \frac{1}{n}\log \P_{u_n}\!\!\left( \sup_{t\in[0,T_{v-u}]} \left| \frac{1}{n} X([nt]) - u - w_{v-u}t \right| < \eps, \; \tau_\vartheta > nT_{v-u} \right) \\ \geq - \langle\alpha_{v-u}, v-u\rangle . 
\end{multline}
Remark finally that  on the event 
\[
\left\{ \sup_{t\in[0,T_{v-u}]} \left| \frac{1}{n} X([nt]) - u - w_{v-u}t \right| < \eps, \; \tau_\vartheta > nT_{v-u}\right\} 
\]
the inequality $|X([T_{v-u} n]) - v| \leq \eps$ holds. From \eqref{eq70-12} it follows therefore that 
\[
\lim_{\eps > 0} \liminf_{n\to\infty} \frac{1}{n}\log \P_{u_n} (|X([T_{v-u} n]) - v| \leq \eps) ~\geq~  -  \langle \alpha_{u-v},u-v\rangle. 
\]
When combined with \eqref{eq70-10} the last relation proves \eqref{eq70-8} 
\end{proof} 

Now, we are ready to prove the lower bound \eqref{eq70-5} for $u\not= v$, $u,v\in{\cal C}$. This is a subject of the following lemma. 

\begin{lemma}\label{lem70-5} Suppose that $u\not= v$, $u,v\in{\cal C}$ and let two sequences $(u_n), (v_n)\in E^\N$ be such that $\lim_n u_n/n = u$ and $\lim_n v_n/n = v$. Then under the hypotheses (C0'), (C1') and (C2'), \eqref{eq70-5} holds. 
\end{lemma} 
\begin{proof} Indeed,  Lemma~\ref{lem70-3} combined with the inequality 
\[
G(u_n,v_n) ~\geq~ \sum_{z\in E: ~ |z - nv| \leq \eps n} \P_{u_n}\left( X([T_{u,v}n]) = z \right) G(z, v_n) \
\]
shows that for any $\eps > 0$,  
\[
G(u_n,v_n) ~\geq~ \delta^{\kappa \eps n} \P_{u_n} \left( |X([T_{u,v} n]) - nv| \leq \eps n \right) 
\]
and hence, using \eqref{eq70-8} one gets 
\begin{align*}
\liminf_{n\to\infty} \frac{1}{n} \log G(u_n,v_n) &~\geq~ \lim_{\eps\to 0} \liminf_{n\to\infty} \frac{1}{n} \log \P_{u_n} \left( |X([T_{u,v} n]) - nv| \leq \eps n \right) \\ 
&~\geq~ - \sup_{\alpha\in D} \langle\alpha, v-u\rangle. 
\end{align*} 
\end{proof} 
Consider now for $u,v\in \overline{\cal C}\times\overline{\cal C}$ 
\[
J(u,v) ~=~ \liminf_{n\to\infty} \frac{1}{n} \log G(u_n,v_n) 
\]
where the limit is taken over all sequences $(u_n), (v_n)\in E^\N$ with $\lim_n u_n/n = u$ and $\lim_n v_n/n = v$. Remark that by Lemma~\ref{lem70-1}, the function $J$ is well defined everywhere on $\overline{\cal C}\times\overline{\cal C}$. To extend the lower bound~\eqref{eq70-5} for $u,v\in\overline{\cal C}\times\overline{\cal C}$, we will use 

\begin{lemma}\label{lem70-6} Under the hypotheses (C0'), (C1') and (C2'), the function $J$ is finite and continuous on $\overline{\cal C}\times\overline{\cal C}$. 
\end{lemma} 
\begin{proof} Let $u,v\in \overline{\cal C}\times\overline{\cal C}$ and let two sequences $(u_n), (v_n)\in E^\N$ be such that $\lim_n u_n/n = u$ and $\lim_n v_n/n = v$. Then by Lemma~\ref{lem70-3}, for any $n\geq 0$ there ix $0 \leq m_n\leq \kappa |u_n-v_n|$ such that 
\[
G(u_n,v_n)  ~\geq~ \P_{u_n}(X(m_n)=v_n) ~\geq~ \delta^{\kappa|v_n-u_n|}
\]
and consequently, 
\[
J(u,v) ~\geq~ \liminf_{n\to\infty} \frac{1}{n} \log \delta^{\kappa|v_n-u_n|} ~=~ \kappa |u-v| \log\delta 
\]
Moreover,  from Lemma~\ref{lem70-2} it follows that for any $\alpha\in D$, 
\[
J(u,v) ~\leq~ - \langle\alpha, v-u\rangle.
\]
The function $J$ is therefore finite on $\overline{\cal C}\times\overline{\cal C}$. To prove its continuity, we use again Lemma~\ref{lem70-3}. Consider $(u,v), (\hat{u},\hat{v})\in \overline{\cal C}\times\overline{\cal C}$ and let sequences $(u_n),(v_n), (\hat{u}_n),(\hat{v}_n)\in E^\N$ be such that 
\[
\lim_{n\to\infty}  u_n/n = u, \quad \lim_{n\to\infty} v_n/n = v,  \quad  \lim_{n\to\infty}  \hat{u}_n/n = \hat{u}, \quad \lim_{n\to\infty}\hat{v}_n/n = \hat{v}, 
\]
and
\[
J(u,v) = \lim_n\frac{1}{n}\log G(u_n,v_n). 
\]
By Lemma~\ref{lem70-3}, for any $n\geq0$ there exist $0\leq k_n\leq \kappa|u_n - \hat{u}_n|$ and $m_n\leq \kappa|v_n-\hat{v}_n|$ such that 
\[
\P_{u_n}(X(k_n) = \hat{u}_n) ~\geq~\delta^{\kappa |u_n-\hat{u}_n|}  \quad \text{and} \quad \P_{\hat{v}_n}(X(m_n) = v_n) ~\geq~\delta^{\kappa |v_n-\hat{v}_n|}. 
\]
Using these inequalities together with  the inequality 
\[
G(u_n,v_n) ~\geq~ \P_{u_n}(X(k_n) = \hat{u}_n) G(\hat{u}_n,\hat{v}_n) \P_{\hat{v}_n}(X(m_n) = v_n) 
\]
we obtain 
\[
\frac{1}{n} \log G(u_n, v_n) ~\geq~ \frac{1}{n} \log G(\hat{u}_n,\hat{v}_n) + \kappa \frac{|u_n-\hat{u}_n| + |v_n-\hat{v}_n|}{n} \log\delta 
\]
and letting finally $n\to\infty$ we conclude that 
\[
J(u,v) ~\geq~ J(\hat{u},\hat{v}) + \kappa(|\hat{u}-u| + |\hat{v}-v|) \log\delta.
\]
Similarly 
\[
J(\hat{u},\hat{v})   ~\geq~ J(u,v) + \kappa(|\hat{u}-u| + |\hat{v}-v|) \log\delta,
\]
and consequently, for any $(u,v), (\hat{u},\hat{v})\in \overline{\cal C}\times\overline{\cal C}$,
\[
|J(\hat{u},\hat{v}) -  J(u,v) | ~\leq~ \kappa(|\hat{u}-u| + |\hat{v}-v|) |\log\delta|
\]
The last inequality proves that the function $J$ is Lipschitz continuous on $\overline{\cal C}\times\overline{\cal C}$ 
\end{proof} 

\medskip

Now we are ready to complete the proof of Proposition~\ref{pr7-100}. Recall that the upper bound~\eqref{eq70-4} is a straightforward consequence of Lemma~\ref{lem70-2}. According to the definition of the function $J :\overline{\cal C}\times\overline{\cal C} \to\R$, to prove the lower bound \eqref{eq70-5} it is sufficient to show that 
\be\label{eq70-13}
J(u,v) ~\geq~ - \sup_{\alpha\in D} \langle \alpha, v-u\rangle 
\ee
for all $(u,v)\in\overline{\cal C}\times\overline{\cal C}$. The last inequality is already proved by Lemma~\ref{lem70-6} for $u\not= v$, $u,v\in{\cal C}$. To extend \eqref{eq70-13} for all $(u,v)\in\overline{\cal C}\times\overline{\cal C}$, it is now sufficient to notice that by Lemma~\ref{lem70-6},  the function $J$ is continuous everywhere on $:\overline{\cal C}\times\overline{\cal C}$ and the function $w \to \sup_{\alpha\in D}\langle\alpha, w\rangle$ is continuous on $\R^d$ (see Corollary~13.2.2 in the book of Rockafellar~\cite{R} ).  Proposition~\ref{pr7-100} is therefore proved.

\subsection{Proof of Proposition~\ref{prC-1}} 
From now on the conditions (C0'), (C1') and (C2) are assumed satisfied. Since for any $x,u\in E$ and $n\geq 1$, 
\[
Q(x, x + nu) \geq Q(0, nu) \quad \text{and} \quad Q(x + nu, x) \geq Q(nu, 0),
\]
to prove Proposition~\ref{prC-1} it is sufficient to show that for any $u\in E$,
\be\label{eq70-14}
\liminf_{n\to\infty} \frac{1}{n}\log Q(0,nu) ~\geq~ 0
\ee
and
\be\label{eq70-15} 
\liminf_{n\to\infty} \frac{1}{n}\log Q(nu,0) ~\geq~ 0
\ee
To get these inequalities we introduce the following truncated and twisted processes.  

Recall that by  (C0'),  there are $\kappa >0$ and a finite set ${\cal E}_0\subset \supp(\mu)=\{x\in\Z^d : \mu(x) > 0\}$, such that for any $x\not= y, \; x,y\in E$, there is a sequence $x_0, x_1, \ldots, x_n\in E$ with $x_0= x$, $x_n = y$ and $n\leq \kappa |y-x|$ such that $x_j-x_{j-1} \in {\cal E_0}$ for all $j\in\{1,\ldots,n\}$. Consider  a sequence of sub probability measures $(\mu_k)$ on $\Z^d$ such that 
\begin{itemize} 
\item[--] $\mu_k(x)\not= 0$ for all $x\in {\cal E}_0$; 
\item[--] for any $x\in\Z^d$, the sequence $\mu_k(x)$ is increasing and tends to $\mu(x)$ as $k\to\infty$; 
\item[--] for any $k\geq 1$, the set $\supp(\mu_k) = \{x\in\Z^d : \mu_k(x) > 0\}$ is finite  and 
\[
\sum_{x\in\Z^d} \mu_k(x) < 1,
\]
\end{itemize} 
Then for any $k\geq 1$ the function 
\[
R_k(\alpha) ~=~ \sum_{x\in\Z^d} e^{\alpha\cdot x} \mu_k(x) 
\]
is strictly convex and finite everywhere in $\R^d$, the set $D_k = \{\alpha\in\R^d: R_k(\alpha) \leq 1\}$ is compact and the gradient 
\[
\nabla \log R_k(\alpha) = \frac{\nabla R_k(\alpha)}{R_k(\alpha))} 
\]
does not vanish on the boundary $\partial D_k = \{\alpha\in\R^d:~R_k(\alpha) = 1\}$ (see~\cite{Hennequin}).  For given $k\geq 1$ and $\alpha_k\in\partial D_k$, we put
\[
\tilde\mu_k(x) ~=~\exp(\langle\alpha_k,x\rangle)\mu_k(x), \quad x\in\Z^d, 
\]
and we introduce a truncated and twisted  random walk $(\tilde{X}_k(n))$ on $E$ with transition probabilities 
\[
\P_x(\tilde{X}_k(1) = y) ~=~ \tilde\mu_k(y-x)
\]
Denote by $\tilde{R}_k$ the generating function of $\tilde\mu_k$:
\[
\tilde{R}_k(\alpha) = \sum_{x\in\Z^d} \exp(\langle \alpha, x\rangle) \tilde\mu_k(x) 
\]
and let $\tilde{D}_k = \{\alpha\in\R^d: \tilde{R}_k(\alpha)\leq 1\}$. The hitting probabilities related to $(\tilde{X}_k(t))$ will be denoted by  
\[
\tilde{Q}_{k}(x,y) ~=~ \P_x( \tilde{X}_k(n) = y, \; \text{for some} \; n > 0), \quad x,y\in E. 
\] 
Remark that the twisted measure $\tilde\mu_k$ is stochastic because for any $x\in\Z^d$, 
\[
\sum_{x\in \Z^d} \tilde\mu_{k} (x) ~=~ \sum_{x\in\Z^d} \mu_k(x) \exp(\langle \alpha_{k} , y-x\rangle ) ~=~ R_k(\alpha_{k}) ~=~ 1.
\]
Moreover, for the twisted random walk $(\tilde{X}_k(n))$ the conditions (C0'), (C1') and (C2') of the previous subsection are clearly satisfied. Using therefore Corollary~\ref{corCC-1} we obtain
\begin{cor}\label{corCC-2}  Under the hypotheses (C0') and (C1') and (C2), for any $k\geq 1$, $u\in E$ and uniformly on $x\in E$, 
\be\label{eq70-16}
\liminf_{n\to\infty} \frac{1}{n}\log \tilde{Q}_k(0, nu) ~\geq~ - \sup_{\alpha\in \tilde{D}_k} \langle \alpha, u\rangle 
\ee
and
\be\label{eq70-17} 
\liminf_{n\to\infty} \frac{1}{n}\log \tilde{Q}_k(nu, 0) ~\geq~ - \sup_{\alpha\in \tilde{D}_k} \langle \alpha, - u\rangle 
\ee
\end{cor} 
\noindent
Remark now that 
\[
\tilde{R}_k(\alpha) = \sum_{x\in\Z^d} \exp(\langle \alpha + \alpha_k, x\rangle) \mu_k(x) ~=~ R_k(\alpha +\alpha_k),
\]
\[
\tilde{D}_k =  \{\alpha\in\R^d: R_k(\alpha+\alpha_k)\leq 1\} = \{\alpha\in\R^d :~R_k(\alpha)\leq 1\} - \alpha_k = D_k - \alpha_k.
\]
and, by definition of the process $(\tilde{X}(t))$, for any $x,y\in E$,
\[
\tilde{Q}_k(x, y) ~\leq~ \exp(\langle \alpha_k, y-x\rangle) Q(x,y).
\] 
Hence, using \eqref{eq70-16} we conclude that for any $u\in E$ and $k\geq 1$, 
\be\label{eq70-18}
\liminf_{n\to\infty} \frac{1}{n}\log Q(0, nu) ~\geq~ - \sup_{\alpha\in D_k} \langle \alpha, u\rangle. 
\ee
Recall now that by construction, the sequence of functions $R_k$ is increasing and by monotone convergence theorem, $R_k(\alpha)\to R(\alpha)$ as $k\to\infty$ for any $\alpha\in\R^d$. 
The sequence of compact sets $D_k$ is therefore decreasing and 
\[
\bigcap_{k\geq 1} D_k ~=~ \{\alpha\in\R^d : R(\alpha) \leq 1\}.
\]
Since by (C2), $\{\alpha\in\R^d : R(\alpha) \leq 1\} = \{0\}$, letting $k\to\infty$ in \eqref{eq70-18} one gets \eqref{eq70-14}. The proof of \eqref{eq70-15} is quite similar and uses the estimates \eqref{eq70-17}. 
Proposition~\ref{prC-1} is therefore proved.

\end{document}